\documentclass[a4paper,11pt, twoside]{article}
\usepackage{amsmath,amsthm}
\usepackage{amssymb,latexsym,booktabs,array}
 \usepackage{enumerate}
\usepackage[colorlinks,citecolor=blue]{hyperref}
\usepackage[numbers,sort&compress]{natbib}
\usepackage{indentfirst}
\usepackage{mathrsfs}
\DeclareMathAlphabet{\mathpzc}{OT1}{pzc}{m}{it}

\newtheorem{theorem}{Theorem}[section]

\newtheorem{lemma}{Lemma}[section]

\newtheorem{remark}{Remark}[section]

\newtheorem{example}{Example}[section]
\theoremstyle{definition} \theoremstyle{remark}
\numberwithin{equation}{section}
\allowdisplaybreaks

\date{}

\begin{document}

\markboth{J. He,J. Huang, F. Su, L. Zhang}{Approximation for stochastic time-space...}

\date{}
\baselineskip 0.22in
\title{{\bf Approximation for stochastic time-space fractional cable equations driven by rough noise}}

\author{Jiawei He$^{1,2}$, Jianhua Huang$^{1}$, Fang Su$^{1,*}$, Lijuan Zhang$^{3}$ \\[1.8mm]
	\footnotesize  {$^1$ College of Science, National University of Defense Technology, Changsha 410073, China}\\
	\footnotesize  {$^2$ School of Mathematics, Guangxi University, Nanning 530004, China} \\
\footnotesize  {$^3$ School of Mathematics, Nanjing Audit University, Nanjing 210000, China}
}

\maketitle

\begin{abstract}
The time-space fractional cable equation arises from extending the generalized fractional Ohm's law to model anomalous diffusion processes. In this paper,  we develop and analyze a numerical approximation  for stochastic nonlinear
time-space fractional cable equation driven by rough noise.  The model features both two nonlocal terms in time and one in space. By an operator theoretic approach, we establish the existence, uniqueness and  regularity of solutions. To regularize the rough noise, we introduce a spectral Wong-Zakai approximation and derive its convergence rate. For the fully discrete scheme, we employ the spectral Galerkin method for spatial discretization and the backward Euler convolution quadrature for temporal discretization, and we derive error estimates under explicit parameter conditions. Finally, numerical experiments are presented to validate the theoretical convergence rates.
\\[2mm]
\end{abstract}

\baselineskip 0.25in
\section{Introduction}\label{}
\setcounter{section}{1} \setcounter{equation}{0}

As is known, the classical cable equation was derived from the Nernst-Planck equation and serves as the most fundamental equation for modeling neuronal dynamics to simulate the electrodiffusion of ions.
 However, when ions exhibit anomalous subdiffusion behavior, the fractional cable equation becomes a powerful mathematical model to incorporate these abnormal diffusion phenomena.
Santamaria et al. \cite{Santamaria} investigated the dynamics of Purkinje cell dendrites influenced by spines, highlighting that this is a dynamic process due to the fact that molecules can both enter and leave spines. To elucidate the anomalous electrodiffusion of ions in spiny dendrites, Henry et al. \cite{Henry,Henry2008} reconstructed a fractional cable equation from the fractional Nernst-Planck equations. By integrating the current continuity equation with the longitudinal current density, Langlands et al. \cite{Langlands} derived the fractional cable equations from  fractional Nernst-Planck equation as macroscopic models for ion electrodiffusion in nerve cells on infinite and semi-infinite domains. These models account for the ionic transmembrane current and any external current traversing the membrane.

It is known that fractional derivatives can be associated with the behavior of a group of particles that are engaged in the continuous-time random walk \cite{Se,Mainardi2010,Metzler}. By substituting a spatial fractional derivative for the Laplacian operator in the diffusion equation, a solution in this equation characterizes the probability density function for particles that are involved in a random walk with heavy tails, a process where infrequent large jumps are more significant than the usual smaller ones. Conversely, a time fractional derivative in time results represent subdiffusion process, consider a situation where the intervals between particle jumps follow a probability distribution with a long tail, and the mean square displacement of a group of particles is proportional to $\eta \in(0,1)$. In this way, when examining the processes of anomalous diffusive transport that stem from temporal memory and spatial non-local effects, a common approach is to modify the constitutive equation (Ohm's law) into a generalized fractional form. By applying this modification to the fractional flux for both the total ionic transmembrane current density and the injected current density,
Li and Deng \cite{Li-Deng} derived the following time-space fractional cable equation
$$
r_mc_m\partial_t V(x,t)=\frac{r_m d}{4r_L}\partial_t^{1-\alpha} \nabla ^{2s} V(x,t)-D(\beta)\partial_t^{1-\beta}(V(x,t)-V_{rest}-r_mi_e(x,t)),
$$
where $r_m$ denotes the specific membrane resistance, $r_L$ signifies the longitudinal resistivity depending on $\alpha\in (0,1),s\in(0,1)$, and $c_m$ is the symbol for membrane capacitance per unit area, and $i_e$ represents the external current injected per unit area. $D(\beta)$ is a parameter depending on $\beta\in(0,1)$, $d$ is the positive constant. $\nabla^{2s}$ is the Riesz fractional operator and $\partial_t^{1-\eta}$ is the Riemann-Liouville fractional derivative, $\eta=\alpha,\beta$ with $0<\beta\leq \alpha<1$.

Observe that the Riesz fractional operator
$\nabla^{2s} =-( \widetilde{\partial}_x^{2s} +\widetilde{\partial}_{-x}^{2s})/(2\cos(\pi s ))$,
where $\widetilde{\partial}_x^{2s}$, $\widetilde{\partial}_{-x}^{2s}$ are the left and right Riemann-Liouville fractional derivatives,  respectively, for instance, we refer to the work of Saichev et al. \cite{Se}, where the Riesz fractional derivative $\nabla^{2s}$
is equivalent to the fractional Laplacian operator $-(-\Delta)^{s}$
when analyzed within the context of Fourier transforms, as they share identical frequency multipliers. However, for the well-posedness analysis and numerical approximation on the bounded domain \(D=(0,l)\) with homogeneous Dirichlet boundary conditions, we work with the spectral fractional Laplacian $A^s e_k=\rho_k^se_k$,
 where \(\{e_k\}\) are the Dirichlet eigenfunctions of the standard Laplacian on \(D\). This operator is standard in the numerical analysis of fractional PDEs on bounded domains, e.g., \cite{Gunzburger18,Nie,Yang2026} and is the one used in our analysis.
 On the bounded domain with Dirichlet data, \(A^s\) is the natural spectral counterpart that preserves the same frequency power-law behaviour while respecting the boundary conditions. We restrict all analytical and numerical results in this work to this spectral definition without considering on the entire space.
 From this point, \(A^s\) denotes the spectral fractional Laplacian associated with the Dirichlet eigenpairs of \(A:=-\Delta\) on \(D\).   This relationship is reasonable to concern with time-space fractional cable equation with nonlinear source term and external fractional Brownian sheet noise
\begin{equation}\label{P}
	\left\{	\begin{aligned}
	 \partial_tu &+\lambda \partial^{1-\beta}_{t} u +\mu \partial^{1-\alpha} _{t}  A^s u =f(u)+\gamma(t) \xi^{H_1,H_{2}} ,~~&&{\rm in}\ D\times (0,T],\\
u &=0,&& {\rm on}\  \partial D \times (0,T],\\
  u &=g ,&& {\rm in}\   D \times\{t=0\},
	\end{aligned} \right.
	\end{equation}
where   $0<\beta\leq \alpha<1$. Let   functions $g\in L^2(D)$, $f$ satisfy the Lipschitz condition with sublinear growth, and $\gamma(\cdot)\in C^1([0,T])$ with $\gamma(0)=0$, $\mu>0$ and $\lambda\geq 0$ are constants. $\partial^{\eta} _{t} $ represents the Riemann-Liouville fractional partial derivative of order $\eta\in(0,1)$ as follows
$$
\partial^{\eta} _{t} y(x,t)=\frac{1}{\Gamma(1-\eta)}\partial_t\int_0^t (t-s)^{-\eta} y(x,s)ds,~~t>0,~~x\in D.
$$
The stochastic term $\xi^{H_1,H_{2}} $ is defined by
$$
\xi^{H_1,H_{2}}(x,t)=\frac{\partial^{2}W^{H_1,H_{2}}(x,t)}{\partial x   \partial t},
$$
with $W^{H_1,H_{2}}(x,t)$ being a fractional Brownian sheet on a stochastic basis
$(\Omega,\mathcal F, (\mathcal F_t)_{t\in I},\mathbb P)$, $I=[0,T]$, such that
$$
\mathbb E(W^{H_1,H_{2}}(x,t)W^{H_1,H_{2}}(y,s))= R_{H_{1}}(x ,y ) R_{H_{2}}(s,t)
 ,
$$
for which  $(x,t),(y,s)\in D\times I$,  and
$$
R_H(x,y)=\frac{|x|^{2H}+|y|^{2H}-|x-y|^{2H}}{2},
$$
where $H_1,H_2\in (0,\frac{1}{2}]$ are spatial Hurst parameters, meaning that the noise is rough in space, and $\mathbb E$ is the expectation acting at $L^2(\Omega,L^2(D))$ as
$$
\mathbb E \|u\|^2=\int_\Omega |u(\omega)|^2 d\mathbb P
<\infty .$$

Various studies have been dedicated to the time fractional cable equations. Notably, the research teams led by Langlands et al. \cite{Henry,Henry2008,Langlands,Langlands-2} have shown that the mathematical treatment of fractional derivatives featuring singular kernels is consistent with the behavior observed in anomalous diffusion processes. This consistency has sparked a surge of interest among researchers in the field of fractional cable equations, with a particular emphasis on the numerical findings that have emerged in recent times.
Liu et al.\cite{Liu11} examined the stability and convergence properties of two implicit numerical methods. Subsequently, Zhuang et al. \cite{Zhuang} proposed the use of the Galerkin  finite element method for the numerical simulation of the fractional cable equation. Zheng and Zhao \cite{Zheng} explored a combination of the discontinuous Galerkin finite element method along the time axis and the Galerkin finite element scheme along the spatial axis. By transforming the time fractional cable equation into an equivalent integral equation, Yang et al. \cite{Yang} conducted a study on the numerical solution and performed a convergence analysis. More recently, Li et al. \cite{LI-Y} applied the Galerkin  finite element method to verify the numerical analysis driven by fractional integrated additive noise. Nevertheless, there are constraints on the outcomes for the time-space case. Li and Deng \cite{Li-Deng} obtained analytical solutions characterized by Green's functions and the asymptotic behaviors of the corresponding fractional moments. Meanwhile, Saxena et al. \cite{Saxena} developed the fundamental solution and its asymptotic behavior, expressed through an infinite series using the Fox-H function.

The stochastic partial differential equations driven by Brownian sheet noise for Hurst parameter have attracted the attention of many scholars, see e.g.  \cite{Bo,Ca,Hu,Hu2,Hong26}.
Specifically, Cao et al. \cite{Cao} explored the numerical results for a class of stochastic evolution equations that incorporate additive white and rough noises. By employing the Wong-Zakai approximation to regularize the noise, they achieved an optimal order of convergence. In this paper, our primary focus lies on the time-space fractional cable equation, which is driven by fractional Brownian sheet noise in both the temporal and spatial directions, with the Hurst parameter being less than or equal to
$1/2$.
By employing an operator theoretic approach, the existence, uniqueness and regularity of solution are first established requiring $2sH_2+(H_1-1)\alpha>0$. In order to generate noise with enhanced smoothness, the Wong-Zakai approximation stands out as a widely adopted approach. Furthermore, a comprehensive analysis regarding the convergence rate of the solution or the regularized equation under this specific approximation form is also presented. Additionally, we use the spectral Galerkin method and backward Euler convolution quadrature method to build the spatial discretization and the temporal discretization, the theoretical results show that the convergence order is also dependent of the parameters $\alpha,$ $s$ and the Hurst indexes.

We now briefly discuss the significance of the three components introduced in (1.1), all of which are associated with anomalous diffusive transport in neuronal dendrites. The two Riemann-Liouville time-fractional derivatives, \(\partial_t^{1-\beta} \) and \(\partial_t^{1-\alpha} \), model memory effects and subdiffusion: experimental evidence indicates that ionic movement in crowded intracellular environments and through ion channels exhibits a mean-square displacement proportional to \(t^\eta\) with \(\eta \in (0,1)\), which is precisely captured by the power-law convolution kernel of the fractional derivative. The spatial fractional operator \(A^s\) may account for long-range spatial jumps, reflecting the fact that electrotonic propagation in branched dendritic trees is non-local and cannot be reduced to nearest-neighbour interactions. The rough fractional Brownian sheet \(\xi^{H_1,H_2}\) with \(H_1,H_2 \le 1/2\) represents temporally and spatially correlated synaptic noise with rough sample paths; the regime \(H<1/2\) corresponds to negative correlations and high-frequency oscillations, which are more realistic for modelling non-Markovian, rapidly fluctuating background synaptic inputs than standard white noise.  This model captures a coupled phenomenon under the memory, spatial non-locality and stochastic perturbation interact concurrently in neuronal anomalous diffusion.

However, there are several limitations persist in the existing literature,  as is known, most deterministic studies focus exclusively on time-fractional cable equations without incorporating spatial non-locality. Another one is that  the stochastic framework in \cite{LI-Y} considers only a single time-fractional derivative driven by a fractional integral noise, thus excluding rough spatial noise and spatial fractional operators. Although the unified rough-noise framework in \cite{Nie,Nie25} handles fractional Brownian noise with general Hurst indices, it does not address the fractional cable equation of double time-fractional coupling present in (1.1). The main contributions of this work are threefold. First, unlike \cite{LI-Y}, the model (1.1) considers the fractional cable equations that incorporates two coupled Riemann-Liouville time-fractional derivatives together with a spectral fractional Laplacian, which characterizes stronger non-locality manner. Second, the operator-theoretic analysis establishes existence, uniqueness, and regularity under the condition  $2sH_2/\alpha+H_1-1>0$, revealing how the rough noise parameters interact with the fractional orders. Third, we provide a numerical analysis including Wong-Zakai regularization, spectral Galerkin spatial discretization, and backward Euler convolution quadrature time discretization, with explicit convergence rates $\mathbb{E}\|u(t_n)-u^N_n\|^2\leq c((N+1)^{-4\sigma}+\tau^{2H_2+\alpha (H_1-1)/s-\varepsilon})$.

The rest of this paper is organized as follows.
In Section \ref{sect2}, we introduce some essential notations  and define the relevant  function spaces. Section \ref{sect3} begins by establishing a fundamental expression for the solution. From this expression, we derive a priori estimates that provide crucial insights into the behavior of the solution. Additionally, we propose an eigenfunction expansion of the solution.
Based on It\^{o} isometry, we then establish the existence and uniqueness of the solution, which subsequently enables us to determine its spatial and temporal regularities.
Section \ref{sect4} focuses on obtaining a convergence result for the regularized equation through the application of the Wong-Zakai approximation.  In Section \ref{sect5}, we delve into the derivation of spatial and temporal discretization schemes, along with corresponding error estimates.
Finally, Section \ref{sect6} presents numerical experiments designed to verify the theoretical results obtained in the preceding sections.

\section{Preliminaries}\label{sect2}
\setcounter{section}{2} \setcounter{equation}{0}
Let $H$ denote the Hilbert space $L^2(D)$ with the norm $\|\cdot\|$ and inner product $(\cdot,\cdot)$.
As is known, the Laplacian operator $A=-\Delta$  on a bounded domain $D\subset \mathbb R^d$  with smooth boundary $\partial D$ and zero Dirichlet boundary value condition,  we have the following spectral problem
\begin{equation}\label{spectral}
A e_k(x)=\rho_ke_k(x),\quad x\in D;\quad e_k(x)=0,\quad x\in\partial D,\quad k\in\mathbb{N},
\end{equation}
where $\{\rho_k\}_{k=1}^\infty$ denotes the set of nondecreasing positive real eigenvalues such that $\rho_k\to+\infty$ as $k\to+\infty$. The corresponding eigenfunctions are defined by $e_k\in H^1_0(D)\cap H^2(D)$ for every $k\in\mathbb{N}$.
The eigenfunctions $e_k$ are normalized so that $\{e_k\}_{k=1}^\infty$ is an orthonormal basis of $H$.

For all $q\ge -1$, denote $\dot{H}^q(D)\subset H^{-1}(D)$ the
Hilbert space induced by the norm as
\[\|u\|_{\dot{H}^q(D)} =\left(\sum_{k=1}^\infty \rho_k ^q|( u,e_k)|^2 \right)^{1/2} .\]
As usual,  $\dot{H}^q(D) $ ($q\ge -1$) forms a Hilbert scale of interpolation spaces,
moreover a Hilbert space $H_0^p(D)$ is interpolation scale  with the $K$-method between $H^1_0(D)$ and $H$ for $p\in [0,1]$ with the norm $\|\cdot\|_{H_0^1(D)}$. Thus,  by interpolation, the $\dot{H}^p(D)$ and $H_0^p(D)$ norms are equivalent for $p\in [0, 1]$. We shall
 use $H$ to denote the space $\dot{H}^0(D)$. We define a fractional power space of $A^{q}$ by $\dot{H}^{q}(D)=D(A^{q/2})$ with norm $\|u\|_{\dot{H}^q(D)}=\|A^{q/2}u\|$. Clearly, $D(A^{1/2})=H_0^1(D)$.

We denote  the fractional Sobolev space of order $s\in(0,1)$ by
$$
H^{s}(D)=\Big\{u\in H;~~|u|_{H^s(D)}^2:=\int_D\int_D \frac{(u(x)-u(y))^2}{|x-y|^{d+2s}}dxdy <\infty\Big\},
$$
and its norm given by $\|\cdot\|_{H^s(D)}=\|\cdot\|+|\cdot|_{H^s(D)}$.
Let $H^s_0(D)$ denote the closure of set $\mathcal D$ in $H^s(D)$, where
$\mathcal D$ denotes the set of  $C^\infty $ functions with compact support in $D$, Its norm is given by
$$
\|u\|_{H^s_0(D)}=\|u\|+c_{d,s}|u|_{H^s(\mathbb{R}^d)} ,
$$
for  $c_{d,s}=(4^ss\Gamma(d/2+s)\pi^{-d/2}/\Gamma(1-s))^{1/2}$.
It is known that $H^s(D)=H^s_0(D)$ for $s\in [0,1/2)$, and $\dot{H}^s(D)=H_0^s(D)$ for $s\in [0,3/2)$.

\section{Regularity of the solutions}\label{sect3}
\setcounter{section}{3} \setcounter{equation}{0}

For $\theta\in (\pi/2,\pi)$ and $\kappa>0$, let $\Sigma_\theta$ be defined by
$$
\Sigma_\theta=\{z\in \mathbb C:\ \ |{\rm arg z}|\leq\theta,~~z\neq 0\},~~\Sigma_{\kappa,\theta}=\{z\in \mathbb C:\ \ |{\rm arg }z|\leq\theta,~~|z|\geq \kappa\}.
$$
The contour $\Gamma_{\kappa, \theta}$ is defined  by
$$
\Gamma_{\kappa, \theta}=\{re^{-i\theta}:~r\geq\kappa\}\cup\{\kappa e^{i\psi}:~|\psi|\leq\theta\}\cup \{re^{ i\theta}:~r\geq\kappa\},
$$
where the circular arc is oriented counterclockwise and the two rays are oriented with
an increasing imaginary part and $i^2=-1$.

\subsection{A basic solution expression}
We recast problem (\ref{P}) with $f\equiv 0$ and $\gamma\equiv0$ into a Volterra integral equation by
$$
u(x,t)+\int_0^t[\lambda k_{ \beta}(t-\tau) u(x,\tau)+\mu k_{ \alpha}(t-\tau)A^su(x,\tau)]d\tau=g(x),
$$
where $k_\varsigma(t)=t^{\varsigma-1}/\Gamma(\varsigma)$ for $\varsigma\in\{\alpha,\beta\}$, $t>0$ and $\Gamma(\cdot)$ is the usual Euler gamma function. It is well known that the operator $-A^s$ generates a semigroup of angle in the
sector $\Sigma_{\pi-\vartheta }$ for $\vartheta \in(\pi/2,\pi)$ having
%
%
$$
\|(z+A^s)^{-1}\|\leq \widetilde{M}/|z|,~~z\in\Sigma_{\pi-\vartheta },
$$
for some constant $\widetilde{M}>0$. By applying the Laplace transform yields
$$
\hat{u}(z)+\lambda z^{-\beta }\hat{u}(z)+\mu z^{-\alpha }A^s\hat{u}(z)=z^{-1}g,
$$
hence $\hat{u}(z)=H(z)g$, with the kernel $H(z)$ given by
$$
H(z)= \mu^{-1}z^{\alpha-1}(h(z)I+A^s)^{-1},~~h(z)=\mu^{-1}z^\alpha(1+\lambda z^{ -\beta}).
$$
By means of the uniqueness of the inverse
Laplace transform, one deduces that the solution operator $S(t)$ is given by
\begin{equation}\label{St}
S(t)=\frac{1}{2\pi i}\int_{\Gamma_{\kappa,\pi-\vartheta}}e^{zt}H(z)dz.
\end{equation}
\begin{lemma}\label{es1}
Let $\vartheta\in(\pi/2,\pi)$, $\lambda\geq 0$, then $h(z)\in\Sigma_{\pi-\vartheta}$ for $z\in \Sigma_{\pi-\vartheta}$ and
$$
|h(z)|^{-1}\leq c\mu \min\{|z|^{-\alpha},~ |z|^{\beta-\alpha } \}.
$$
\end{lemma}

\begin{proof}

Let $z\in \Sigma_{\pi-\vartheta}$, i.e., $z=re^{i\psi}$ for $|\psi|<\pi-\vartheta$, $r>0$, then
$$
h(z)=\mu^{-1}r^\alpha e^{i\alpha\psi}+\mu^{-1}\lambda r^{\alpha-\beta} e^{i(\alpha-\beta)\psi}\in \Sigma_{\pi-\vartheta},
$$
since $0<\beta\leq\alpha<1$. For $\lambda>0$, we note that
$$
|z^\beta+\lambda|^2=\lambda^2+2\lambda r \cos(\beta\pi)+r^{2\beta}>\lambda^2+2\lambda r\cos(\beta\psi)+r^{2\beta},
$$
and function $k(x)=\lambda^2+2\lambda \cos(\beta\psi)x+x^{2\beta}$, $x\geq0$ attains its minimum at $x=0$, then
$|z^\beta+\lambda|^2>\lambda^2$,
it follows from $\lambda>0$ and $\sin(\beta\psi)>0$ that
$$
|h(z)|^{-1}=\frac{\mu  }{|z^{\alpha-\beta}||z^\beta+\lambda|}
< \frac{\mu|z|^{\beta-\alpha}}{\lambda }.
$$
Moreover, we have
$$
|z^\beta+\lambda|^2>(\lambda+r^\beta\cos(\beta\psi))^2+(r^\beta\sin(\beta\psi))^2\geq (r^\beta\sin(\beta\psi))^2 ,
$$
which shows that
$$
|h(z)|^{-1} =\frac{\mu}{|z^{\alpha-\beta}||z^\beta+\lambda|}
<\frac{\mu }{ r^{\alpha }  \sin(\beta\psi) }=\frac{\mu |z|^{- \alpha}}{ \sin(\beta\psi) }.
$$
Clearly, for $\lambda=0$, $|h(z)|^{-1}=\mu  |z|^{-\alpha}$. The estimate still holds.
This ends the proof.
\end{proof}

\subsection{A priori estimate}

We first show the regularity to problem (\ref{P}) with $f\equiv 0$ and $\gamma\equiv0$.

\begin{theorem}\label{Thm1}
For any $g\in H$, $s\in(0,1)$, and $f\equiv 0$ and $\gamma\equiv0$, then there exists a unique solution $u$ to
problem (\ref{P}) satisfying
$$
u=S(t)g\in C([0,T];H)\cap C^1((0,T];H^{2s}(D)).
$$
Moreover, the stability estimates hold for $t\in (0, T ]$ and $\nu = 0, 1$, $m=1,2,\cdots$ as
$$
\|A^{s\nu} S^{(m)}(t)g\|\leq c t^{-m-\alpha\nu }\|g\|,~~~g\in H.
$$
$$
\|A^{s\nu}  S^{(m)}(t)g\|\leq c t^{-m+(1-\nu) \alpha  }\|A^s g\|,~~~g\in D(A^s).
$$
\end{theorem}

\begin{proof}
By Lemma \ref{es1}, we have
$$
\|(h(z)I+A^s)^{-1}\|\leq \widetilde{M}/|h(z)|, ~~~z\in\Sigma_{\pi-\vartheta},
$$
and we deduce   that
$$
\|H(z)\|\leq c\widetilde{M} \min\{ |z|^{-1},~|z|^{\beta-1}\}~~~z\in\Sigma_{\pi-\vartheta}.
$$
In view of \cite[Theorem 2.1]{Pruss},  there exists a unique
solution $u$ which is given by
$$
u(t)=S(t)g,~~g\in H.
$$

Now, let $\kappa=1/t$ for $t>0$. For $\nu=0$ and $m\in\mathbb N$, from $\|H(z)\|\leq c/|z|$ we have
\begin{align*}
\|S^{(m)}(t)\|=&\left\| \frac{1}{2\pi i}\int_{\Gamma_{1/t,\pi-\vartheta}}z^me^{zt}H(z)dz\right\|
\leq    \frac{1}{2\pi  }\int_{\Gamma_{1/t,\pi-\vartheta}}|z|^m e^{{\rm Re}(z)t}\|H(z)\| |dz|\\
\leq &c\left(\int_{1/t}^\infty  r^{m-1} e^{-rt\cos(\vartheta)}dr+\int_{-\pi+\vartheta}^{\pi-\vartheta} e ^{\cos(\psi)}t^{-m}d\psi\right)\leq  c t^{-m}.
\end{align*}
For $\nu=1$ and $m\in\mathbb N$, we note that
$$
A^sH(z)=\mu^{-1}z^{\alpha-1}-h(z)H(z),
$$
it follows  that
$$
\|A^sH(z)\|\leq  \mu^{-1}|z|^{\alpha-1}+\|h(z)H(z)\|\leq c |z|^{\alpha-1}.
$$
We thus obtain that
\begin{align*}
\|A^s S^{(m)}(t)\|=& \left\| \frac{1}{2\pi i}\int_{\Gamma_{1/t,\pi-\vartheta}}z^me^{zt}A^sH(z)dz\right\|\\ \leq  &
c\left(\int_{1/t}^\infty  r^{m+\alpha-1} e^{-rt\cos(\vartheta)}dr+\int_{-\pi+\vartheta}^{\pi-\vartheta} e ^{\cos(\psi)}t^{-m-\alpha}d\psi\right)\leq c t^{-m-\alpha}.
\end{align*}

 For $\nu=0$ and $m\in\mathbb N$, we notice that
\begin{align*}
 S^{(m)}(t)g =&  \frac{1}{2\pi i}\int_{\Gamma_{1/t,\pi-\vartheta}}z^me^{zt}H(z)gdz \\
=& \frac{\mu}{2\pi i}\int_{\Gamma_{1/t,\pi-\vartheta}}z^{m+\alpha-1}h(z)^{-1}e^{zt}gdz- \frac{1}{2\pi i}\int_{\Gamma_{1/t,\pi-\vartheta}}z^{m }h(z)^{-1}e^{zt}H(z)A^s gdz,
\end{align*}
where we have used 
$$
h(z)A^{-s}(h(z)+A^s)^{-1}=A^{-s}-(h(z)+A^s)^{-1}.
$$
Observe that
$$\frac{\mu}{2\pi i}\int_{\Gamma_{1/t,\pi-\vartheta}}z^{m -1}h(z)^{-1}e^{zt}gdz=0,$$ for $m\geq 1$, and then
\begin{align*}
 \frac{\mu}{2\pi i}\int_{\Gamma_{1/t,\pi-\vartheta}}z^{m+\alpha-1}h(z)^{-1}e^{zt}gdz= \frac{\mu}{2\pi i}\int_{\Gamma_{1/t,\pi-\vartheta}}z^{m+\beta-1}(z^\beta+\lambda)^{-1}e^{zt}gdz=0,
\end{align*}
since $|z^\beta+\lambda|>\lambda$ from Lemma \ref{es1}, and there is no singular point at counter $\Gamma_{1/t,\pi-\vartheta}$. It follows from $|h(z)|^{-1}\|H(z)\|\leq c|z|^{ -\alpha-1}$ that
\begin{align*}
 \|S^{(m)}(t)g\|\leq &  \frac{1}{2\pi i}\int_{\Gamma_{1/t,\pi-\vartheta}}|z|^{m }|h(z)|^{-1}e^{{\rm Re }(z)t}\|H(z)\||dz| \|A^s g\|\\
\leq  &
c\left(\int_{1/t}^\infty  r^{m+  \alpha-1} e^{-rt\cos(\vartheta)}dr+\int_{-\pi+\vartheta}^{\pi-\vartheta} e ^{\cos(\psi)}t^{-m+\alpha  }d\psi\right)\|A^s g\|\\ \leq& c t^{-m+\alpha  }\|A^s g\|.
\end{align*}
The same way with $\nu= 1$ replaced $g$ by $A^s g$ in the case $\nu=0$.

\end{proof}

We note that substituting $|h(z)|^{-1}\|H(z)\|\leq c|z|^{2\beta -\alpha-1}$ in the proof of Theorem \ref{Thm1}, it yields $\|A^{s\nu}  S^{(m)}(t)g\|\leq c t^{-m+(1-\nu) (\alpha-2\beta)}\|A^s g\|$, for $g\in D(A^s).$
In particular, from the interpolation property, we have
\begin{equation}\label{inter}
\|A^{s\eta}H(z)\|\leq c |z|^{\alpha\eta-1},~~z\in\Sigma_{\pi-\vartheta},~\eta\in[0,1].
\end{equation}
Hence, we also have the following estimate.
\begin{remark}
For any $g\in \dot{H}^p(D) $, and $f\equiv 0$ and $\gamma\equiv0$, the solution $u=S(t)g$ to problem (\ref{P}) satisfies
$$
\|S^{(m)}(t)g\|_{\dot{H}^q(D)}\leq ct^{-m-\frac{\alpha(q-p)}{2s}}\|g\|_{\dot{H}^p(D)}
$$
for $0\leq p\leq q\leq 2s$ whenever $m=0$ or $0\leq p,q\leq 2s$ whenever $m\geq1$.
\end{remark}

\begin{remark}
For $\nu=1$ and $m=0$, we know that the bound in Theorem \ref{Thm1} is sharp for $t>0$ since bound $\|A^sH(z)\|\leq c|z|^{\alpha-1}$, that is
$$
\|A^s S(t)g\|\leq ct^{-\alpha}\|g\|,~~~g\in H,~~t>0.
$$
On the other hand, a similar proof in Theorem \ref{Thm1} shows from $\mathcal L(  k_{ \varsigma}* \frac{d}{dt} S   (t)  )(z) =z^{1-\varsigma}H(z)$ the Laplace transform of $k_{ \varsigma}* \frac{d}{dt} S (t)$ that $
\| k_{ \varsigma}* \frac{d}{dt} A^{s\nu} S   (t)\|\leq c t^{\varsigma-1-\nu\alpha}$, $t>0,$ for $\varsigma=\beta $, $\nu=0 $ and $\varsigma=\alpha $, $\nu=1$, where $*$ stands for the convolution in $t$.
This bound together with $\nu=0$ and $m=1$, the derivative
$$k_{ \varsigma}* \frac{d}{dt} A^{s\nu} S (t)g=\frac{d}{dt} k_{ \varsigma}* A^{s\nu} S  ( t)g-k_{ \varsigma}(t)g,$$  imply a priori estimate for the solution under the $f\equiv 0$ and $\gamma\equiv0$ as  follows
$$
t\|\partial_tu\|+t^{1-\beta }\|\partial_t^{1-\beta}u\| +t^{1-\alpha} \|\partial_t^{1-\alpha} u\|_{\dot{H}^s(D)}\leq c\|g\|,~~~{\rm for }~~t>0.
$$

\end{remark}
\subsection{Eigenfunction expansion of solution}
By using eigenfunction expansion, the solution of
problem (\ref{P}) can rewrite as
\begin{align*}
u(x,t)=&\int_D F(t,x,y)g(y)dy+\int_0^t \int_D F(t-\tau,x,y)f(u)d\tau dy\\ &+\int_0^t \int_D F(t-\tau,x,y)\gamma(\tau)\xi^{H_1,H_2} (dy,d\tau),
\end{align*}
where $F(t,x,y)$ is defined by
$$
F(t,x,y)=\sum_{k=1}^\infty u_k(t )e_k(x)e_k(y),
$$
and
$$
 u_k(t )=\frac{1}{2\pi i}\int_{\Gamma_{\kappa,\pi-\vartheta}}e^{zt}\mu^{-1}z^{\alpha-1}(h (z) + \rho_k^s)^{-1} dz,
$$
is the unique solution to the following equation:
$$
u_k'(t)+ \lambda \partial^{1-\beta}_{t} u_k(t) +\mu \rho_k^s\partial^{1-\alpha} _{t}  u_k(t)=0,~~~u_k(0)=1.
$$
For a continuous function $u(t)$ for $t>0$ with finite $u(0)$, from the Laplace transform $\mathcal L(v)=\hat{v}$ of fractional derivative
$$\mathcal L(\partial_t^{1-\varsigma} u)(z)=z^{1-\varsigma}\hat{u}(z),$$
 for $\varsigma\in\{\beta,\alpha\}$, we get
$$
\hat{u}_k(z)=\frac{z^{\alpha-1}}{z^\alpha+\lambda z^{\alpha-\beta}+\mu\rho_k^s }.
$$

In view of the inverse Laplace transform, we let
	$$u_k(t):= \frac{1}{{2\pi i}}\int_{Br} {{e^{zt}}\frac{z^{\alpha-1}}{z^\alpha+\lambda z^{\alpha-\beta}+\mu\rho_k^s } d z}, $$
where $Br = \{ z:{\rm Re} \,z = \sigma ,\sigma  > 0\} $ is the Bromwich path.

To obtain the integral representation of \( u_k \), we can transform the Bromwich path of integration \( Br \) into the equivalent Hankel path \( Ha \). Note that the integrand function in the integral has a branch point at zero; thus, we cut off the negative part of the real axis.
Furthermore, let $z= r{e^{i \theta }}$, for $r > 0, \theta  \in ( \pi/2,\pi )$, we have
\begin{align*}&\Im \{z^{1-\alpha}(z^\alpha+\lambda z^{\alpha-\beta}+\mu\rho_k^s)\} \\
 =& r \sin \left(\theta\right)  +\lambda {r^{1-\beta} }\sin\left((1- \beta )\theta \right)+\mu\rho_kr^{1-\alpha}\sin((1-\alpha)\theta) \ne 0,
 \end{align*}
for $\theta  \ne 0$  and $\lambda \geq0$, and $\sin (\alpha\theta) $, $\sin ((\alpha-\beta) \theta) $ and $\sin((1-\alpha)\theta)$ have the same sign, this allows
 $z^\alpha+\lambda z^{\alpha-\beta}+\mu\rho_k^s$ admitting no zero in the main sheet of the Riemann surface including its boundaries on the cut. Consequently, the function \( u_k(t) \) can be determined by deforming the Bromwich path into the Hankel path \( Ha(\varepsilon) \). This path starts from \( -\infty \) along the lower side of the negative real axis, encircles the disk \( |z| = \varepsilon \) counterclockwise, and ends at \( -\infty \) along the upper side of the negative real axis. As \( \varepsilon \to 0 \), we obtain
		\begin{equation*}
			\begin{aligned}	
				{u_k}(t) 
				=& \frac{1}{\pi }\int_0^\infty  e^{ - rt}r^{\alpha-1}
                \frac{a(r)\sin( \alpha\pi)-b(r)\cos( \alpha\pi)}{ a^2 (r)+b^2 (r)} d r\\
                =& \frac{1}{\pi }\int_0^\infty  e^{ - rt}r^{\alpha-1}
                K_{k}(r) d r,
			\end{aligned}
		\end{equation*}
where
$$
K_k(r)=\frac{\lambda r^{\alpha-\beta}\sin(\beta\pi)+\mu\rho_k^s\sin(\alpha\pi)}{ a^2 (r)+b^2 (r)}>0,
$$
and
$$
a(r)= r^\alpha\cos ( \alpha \pi) + \lambda {r^{\alpha-\beta} }\cos ((\alpha-\beta)\pi) +\mu {\rho_k^s },
$$
$$
b(r)=r^\alpha \sin (\alpha\pi)+\lambda r^{\alpha-\beta}\sin((\alpha-\beta)\pi).
$$

By applying the Laplace transform properties to $
\hat{u}_k$, we obtain
    $$u_k(0) = \mathop {\lim }\limits_{z  \to  + \infty }z{{\hat u}_k}(z ) = 1.$$
Recall that a function $u(t)$ is said to be completely monotone if and only if
$$
(-1)^nu^{(n)}(t)\geq0,~~{\rm for~all}~t\geq0,~~n=0,1,\cdots.
$$
 \begin{lemma}\label{lemma-3.1}
For $t\geq 0$, the functions $u_k(t)$ are continuous for each $k=1,2, \cdots $, and they possess the properties:
		\begin{enumerate}
    \item [{\rm (i)}] For all $t \geq 0$,  $0 < {u_k}(t) \leq  {u_k}(0)=1 $.
     \item [{\rm (ii)}]  ${u_k}(t)$ is completely monotone for all $t\geq0$.
    \item [{\rm (iii)}] For all $t\geq 0$, $\lambda\geq 0$, there exist constants $c_2\geq 1,c_1>0$ such that
		\begin{equation*}
			c_1e^{-t}\leq \rho_k^s{u_k}(t) \leq  \frac{c_2\rho_k^s}{1+\lambda t^{ \beta}+\rho_k^s t^\alpha}.
		\end{equation*}	
 	\item [{\rm (iv)}] For all $0<  t < \infty$, $\eta\in [0,1]$ there holds $\left|u_k'(t)\right|
 	\leq c \rho_k^{-s\eta} t^{-\alpha\eta-1}.$ Moreover, for $h\in(0,t)$, $\sigma\in[0,1]$, it holds
		$$\left| {{u_k}(t) - {u_k}(t-h)} \right| \leq c \rho_k^{-s\eta} h^\sigma (t-h)^{-\sigma-\alpha\eta}.$$
\end{enumerate}
\end{lemma}

\begin{proof}
From the definition of $u_k(t)$ with Hankel contour, it is obvious that ${u_k}(t)$ is completely monotone for $t> 0$. Then for each $k \in \mathbb{N}$, ${u_k}(t)$ is strictly monotonically decreasing for $t \geq 0$, which implies $0 < {u_k}(t) \leq {u_k}(0).$ Hence,  the first assertion follows directly from the properties of the Laplace transform.

The integrated function $K_k(r)$ is positive and bounded for all $r\geq0$. Indeed, according to
$ a(r)\sin(\alpha\pi)+b(r)\cos(\alpha\pi)=\lambda r^{\alpha-\beta}\sin(\beta\pi)+\sin(\alpha\pi)\mu\rho_k^s>0,
$
we assert that
\begin{align*}
(\lambda r^{\alpha-\beta}\sin(\beta\pi)+\sin(\alpha\pi)\mu\rho_k^s)^2
\leq& a^2(r)+b^2(r),
\end{align*}
by applying to the definition of $K_k$. We thus deduce that
$$
0<K_k(r)\leq \frac{1}{\lambda r^{\alpha-\beta}\sin(\beta\pi)+\sin(\alpha\pi)\mu\rho_k^s}.
$$
For all $t > 0$, still applying the above estimate, for $\lambda>0$ we obtain that
\begin{align*}
{u_k}(t) \leq &
\frac{1}{\pi }\int_0^\infty  e^{ - rt}r^{\alpha-1}
\frac{1}{ \lambda r^{\alpha-\beta} \sin(\beta\pi)+\sin(\alpha\pi)\mu\rho_k^s} d r\\
\leq &\min\left\{ \frac{\Gamma(\beta)}{ \lambda\pi \sin(\beta\pi)}t^{-\beta}, \frac{\Gamma(\alpha)}{\sin(\alpha\pi)\pi \mu\rho_k^s}t^{-\alpha}\right\}. 	
\end{align*}
The case of $\lambda=0$ is obvious. Therefore, together with $0<u_k(t)\leq 1$ imply
$$
0<u_k(t)\leq \frac{c_2}{1+\lambda t^\beta+\rho_k^s t^\alpha}, \ \ t\geq0.
$$

On the other hand, we get
$$
a^2(r)+b^2(r) \leq 3(r^{2\alpha}+\lambda^2r^{2(\alpha-\beta)}+\mu^2\rho_k^{2s}).
$$
The left hand side estimate of $u_k$ reads
\begin{align*}
u_k(t)\geq &\frac{1}{3\pi}\int_0^\infty e^{-rt}r^{\alpha-1}\frac{\mu\rho_k^s\sin(\alpha\pi)}{r^{2\alpha}+\lambda^2r^{2(\alpha-\beta)}+\mu^2\rho_k^2}dr\\
\geq & \frac{\sin(\alpha\pi)}{3\mu\rho_k^s\pi}\int_0^\infty e^{-rt}r^{\alpha-1}\frac{1}{\frac{r^{2\alpha}+\lambda^2r^{2(\alpha-\beta)}}{\mu^2\rho_1^{2s}}+1}dr\geq \frac{c_1}{\rho_k^s}e^{-t},
\end{align*}
where we have used
\begin{align*}
\int_0^\infty e^{-rt}r^{\alpha-1}\frac{1}{\frac{r^{2\alpha}+\lambda^2r^{2(\alpha-\beta)}}{\mu^2\rho_1^{2s}}+1}dr
\geq &e^{-t}\int_0^1 \frac{\mu^2\rho_1^2}{ r^{2\alpha}+\lambda^2r^{2(\alpha-\beta)} +\mu^2\rho_1^{2s}}dr
=: \frac{3c_1\mu\pi\rho_k^s}{\sin(\alpha\pi)}e^{-t}.
\end{align*}
Then the third assertion follows.
By the interpolation from (\ref{inter}), we get
\begin{equation}\label{inter2}
|(h(z)+\rho_k^s)^{-1}|\leq c\rho_k^{-s\eta}|z|^{-\alpha(1-\eta)},~~z\in\Sigma_{\pi-\vartheta},~~\eta\in[0,1],
\end{equation}
this together with
\begin{equation*}
\begin{aligned}	
u_k'(t)
        =&  \frac{1}{2\pi i }\int_{\Gamma_{\kappa,\pi-\vartheta}} e^{z t} \mu^{-1}z^\alpha (h(z)+\rho_k^s)^{-1}dz,
			\end{aligned}
		\end{equation*}
imply that for $t>0$
\begin{equation*}
			\begin{aligned}
				\left|u_k'(t)\right|
			\leq c\rho_k^{-s\eta}\int_{\Gamma_{\kappa,\pi-\vartheta}} |e^{z t}| |z|^{\alpha\eta}  |dz|\leq c \rho_k^{-s\eta} t^{-\alpha\eta-1}.
			\end{aligned}
\end{equation*}
and then
\begin{equation*}
	\begin{aligned}
\left| {{u_k}(t) - {u_k}(t-h)} \right| \leq \int_{t-h}^t |u'_k(\tau)|d\tau\leq  c \rho_k^{-s\eta}\int_{t-h}^t \tau^{-\alpha\eta-1}d\tau.
	\end{aligned}
\end{equation*}
where $h\in (0,t)$. Moreover, for $\sigma\in[0,1]$,  we get
\begin{equation*}
	\begin{aligned}
\left| {{u_k}(t) - {u_k}(t-h)} \right| \leq & c \rho_k^{-s\eta}\int_{t-h}^t \tau^{-\alpha\eta-1}d\tau\leq c \rho_k^{-s\eta}(t-h)^{-\sigma-\alpha\eta}(t^\sigma-(t-h)^{\sigma}),
	\end{aligned}
\end{equation*}
from the fundamental inequality $
		\xi_1^\sigma-\xi_2^\sigma \leq\left(\xi_1-\xi_2\right)^\sigma$, for $0 \leq \xi_2 \leq \xi_1<\infty$, then  (iv) follows.
\end{proof}

By an abstraction process in $x$ for
$$
S(t)v(x)=\int_D F(t,x,y)v(y)dy,
$$
 the solution can also be written as
$$
u( t)= S(t )g +\int_0^t S(t-\tau )f(u)d\tau+\int_0^t \int_D F(t-\tau,x,y)\gamma(\tau)\xi^{H_1,H_2} (dy,d\tau) .
$$
\subsection{Existence and uniqueness}
Let $c>0$ be constant, suppose that nonlinear $f$ possesses the following assumptions
\begin{equation}\label{f}
\begin{aligned}
\|f( u)\|\leq c(1+\|u\|),\\
\|f(u)-f(v)\|\leq c\|u-v\|.
\end{aligned}
\end{equation}

The next lemma is the It\^{o} isometry of fractional Brownian sheet noise on one dimension.
\begin{lemma}{\rm\cite{Bardina}}
	Let $r_1(x,t)=r_{1,1}(x)r_{1,2}(t)$, $r_ 2(x,t)=r_{2,1}(x)r_{2,2}(t)$ satisfying
	$r_{i,1}(x)\in H_0^{\frac{1-2H_1}{2}}(D)$ and $r_{i,2}(t)\in H_0^{\frac{1-2H_2}{2}}(0,T)$, $i=1,2$. Then,
$$
\mathbb E\left(\prod_{i=1}^2\int_0^T\int_D r_i(x,t)\xi^{H_1,H_2} (dx,dt)  \right)= (L_{H_{1},x}r_{1,1}(x),r_{2,1}(x))_{\mathbb R}  (L_{H_{2},t}r_{1,2}(t),r_{2,2}(t))_{\mathbb R},
$$
where
$$
L_{H_1,x}u(x)=\left\{
\begin{aligned}
	&c_{H_1} P.V. \int_{\mathbb R} \frac{u(x)-u(y)}{|x-y|^{2-2H_1}}d\tau ,&& H_1\in(0,\frac{1}{2}),\\
	&u(x), && H_1=\frac{1}{2},
\end{aligned}
\right.
$$
and
$$
L_{H_2,t}u(t)=\left\{
\begin{aligned}
	&c_{H_2} P.V. \int_{\mathbb R} \frac{u(t)-u(\tau)}{|t-\tau|^{2-2H_2}}d\tau ,&& H_2\in(0,\frac{1}{2}),\\
	&u(t), && H_2=\frac{1}{2},
\end{aligned}
\right.
$$
for $c_{H_i}=H_i(1-2H_i)$, $i=1,2$ and $P.V.$ stands the principal value integral.
\end{lemma}
\begin{lemma}\label{xi-H}{\rm\cite{Nie25}}
Let $r_1(x,t)=r_{1,1}(x)r_{1,2}(t)$ and $r_ 2(x,t)=r_{2,1}(x)r_{2,2}(t)$ satisfying
$r_{i,1}(x)\in H_0^{\frac{1-2H_1}{2}}(D)$ and $r_{i,2}(t)\in H_0^{\frac{1-2H_2}{2}}(0,T)$, $i=1,2$.. Then,
$$
\mathbb E\left(\prod_{i=1}^2\int_0^T\int_D r_i(x,t)\xi^{H_1,H_2} (dx,dt)  \right)\leq c \prod_{i=1}^2 \|\partial_t^{\frac{1-2H_2}{2}}r_{i,2}(t)\|_{L^2(0,T)}\|r_{i,1}(x)\|_{H_0^\frac{1-2H_1}{2}(D)}.$$
\end{lemma}

\begin{lemma}\label{deri}
	Let $\eta\in[0,1]$ and $\varsigma\in [0,1]$ such that $\varsigma+ \alpha\eta<1/2$, then for every $t\in(0,T]$
	$$
	\left\|\partial_\upsilon^{\varsigma} \rho_k^{s\eta} u_k(t-\upsilon) \gamma(\upsilon)\right\| _{L^2(0,t)} \leq c.
	$$
\end{lemma}

\begin{proof}
	The  $\gamma\in C^1(I)$ with  $\gamma(0)=0$ implies that
	\begin{align*}
		\int_0^t \left|\partial_\tau^{\varsigma} u_k(t-\tau)\gamma(\tau)\right|^2 ds= \int_0^t \left| I_\tau^{1-\varsigma } (-u_k'(t-\tau)\gamma(\tau)+u_k(t-\tau)\gamma'(\tau))\right|^2 d\tau .
	\end{align*}
	By (\ref{inter2}), $|u_k'(t)|\leq c\rho_k^{-s\eta}(t)^{-\alpha\eta-1}$ for all $\eta\in[0,1]$, $t>0$,
	let $$ u_{tk}(\tau):=-u_k'(t-\tau)\gamma(\tau)+u_k(t-\tau)\gamma'(\tau),$$ a.e. $\tau\in[0,t]$,
	together with $|u_k(t)|\leq c\rho_k^{-s\eta} t ^{-\alpha\eta}$ by interpolation, and taking $\gamma_\infty= \max_{t\in I}(|\gamma(t)|+|\gamma'(t)|)$,
	imply
	$$| u_{tk}|\leq c\rho_k^{-s\eta}(t-\tau)^{-\alpha\eta-1}, ~~~{\rm for}\ t>\tau.$$
It follows that
 \begin{align*}
 	& \int_0^t \left| I_\tau^{1-\varsigma } u_{tk}(\tau)\right|^2 d\tau
 	\leq  c^2\rho_k^{-2s\eta}\int_0^t \left( \int_0^\tau k_{1-\varsigma }(\tau-\upsilon)   (t-\upsilon)^{-\alpha\eta-1}  d\upsilon\right)^2 d\tau ,
 \end{align*}
therefore the Cauchy-Schwarz inequality and Fubini theorem show that
%
	\begin{align*}
	&\int_0^t \left( \int_0^\tau k_{1-\varsigma }(\tau-\upsilon)  (t-\upsilon)^{-\alpha\eta-1} d\upsilon\right)^2 d\tau\\
	\leq & c\int_0^t  \int_\upsilon^t k_{1-\varsigma}^2(\tau-\upsilon)  (t-\upsilon)^{-\alpha\eta-1}   (t-\tau)^{-\alpha\eta} d\tau d\upsilon
	\leq c t^{1-2(\varsigma+\alpha\eta)} ,
\end{align*}
in view of $\varsigma+\alpha\eta<1/2$ for some $\eta\in [0,1]$, where the constant $c$ is independent of $k$. Therefore, the desired estimate holds.
\end{proof}
\begin{remark}\label{Re-2}
	From Lemma \ref{deri}, we know that for every
$t\in(0,T]$ and $\varsigma\in(0,1/2)$, $I_\tau^{1-\varsigma } u_{tk}(\tau)$ belongs to $L^2(0,t)$. According to  \cite[Lemma 2.5]{Li-Xie}, since
$y_{tk}(\tau):=u_{ k}(t-\tau)\gamma(\tau)$ for $\tau\in[0,t]$ and $y_{tk}(\tau)\in L^2(0,t)$, it follows that $\partial_\tau^{ \varsigma  } y_{tk}(\tau) $
	is an element of $H^{ \varsigma}(0,t)$. Additionally, the norm in the space
$H^\varsigma(0,t)$ is given by
	$$
\|y_{tk}(\tau)\|_{H^\varsigma(0,t)}= \left\|\partial_\tau^{\varsigma}  y_{t,k}(\tau)\right\| _{L^2(0,t)} .
$$
\end{remark}
%
%

\begin{theorem}\label{thm2}
Let $f$ satisfy the assumptions (\ref{f}), and   $2s H_2/\alpha+H_1-1>0 $, then there is a unique solution $u$ on $C(I;L^2(D,H))$.
\end{theorem}
\begin{proof}
Let us introduce an operator $P$ on $C(I;L^2(D,H))$ with norm
$$\|u\|_c=\sup_{t\in I}( \mathbb E\|u(t)\|^2)^{1/2},$$
defined by
$$
Pu(t)=S(t )g +\int_0^t S(t-s )f(u)ds+\int_0^t \int_D F(t-s,x,y)\gamma(s)\xi^{H_1,H_2} (dy,ds).
$$
Theorem \ref{Thm1} shows that $Pu\in C(I;L^2(D,H))$. By the Cauchy-Schwarz inequality,
\begin{align*}
 \mathbb E \|Pu(t)\|^2\leq &  3 \mathbb E \|S(t )g\|^2+ 3\mathbb E \left\| \int_0^t S(t-\tau )f(u)d\tau\right\|^2\\
 &+3\mathbb E \left\|\int_0^t \int_D F(t-\tau,x,y)\gamma(\tau)\xi^{H_1,H_2} (dy,d\tau)\right\|^2\\
 \leq& 3c^2\mathbb E \|g\|^2+3c^2t \int_0^t (1+\mathbb E \left\| u(\tau)\right\|^2) d\tau+3Q.
\end{align*}

As for the estimate $Q$, by Lemma \ref{xi-H}, we have
\begin{align*}
Q\leq& c   \sum_{k=1}^\infty \int_0^t \left|\partial_\tau^{\frac{1-2H_2}{2} } u_k(t-\tau)\gamma(\tau)\right|^2 d\tau\|e_k(y)\|_{H_0^{\frac{1-2H_1}{2}}(D)}^2\|e_k(x)\|^2.
\end{align*}
Together Lemma \ref{deri} and (\ref{inter2}) imply
\begin{align*}
	\int_0^t \left|\partial_\tau^{\frac{1-2H_2}{2} } u_k(t-\tau)\gamma(\tau)\right|^2 d\tau\leq c\rho_k^{-2s\eta} ,
\end{align*}
for replacing $\varsigma=\frac{1-2H_2}{2}$ and requiring $H_2>\alpha\eta$.
Since the $k$-th eigenvalue of the Dirichlet boundary problem for the Laplace operator $A=-\Delta$  in bounded  domain $D=(0,l)$ is $\rho_k\simeq k^2$,~~$k=1,2,\cdots,$ it follows that
$$	Q\leq  c   \sum_{k=1}^\infty \rho_k^{-2s\eta+\frac{1-2H_1}{2}} <\infty,$$
where we need to require $2s\eta>1-H_1$ such that $2s\eta=1-H_1+\epsilon$ with small $\epsilon\in(0,2s\eta+H_1-1)$ for some $\eta\in(\frac{1-H_1}{2s},\frac{H_2}{\alpha})\cap [0,1]$.

Based on the above estimates, we have verified that \(\sup_{t \in I} \mathbb{E} \|Pu(t)\| < +\infty\) for any \(u \in C(I, L^2(D, H))\). We now proceed to verify that \(P\) has a fixed point. For any \(u_1, u_2 \in C(I, L^2(D, H))\), we have
\begin{align*}
	\mathbb E \|Pu_1(t)-Pu_2(t)\|^2\leq &  t \int_0^t  \mathbb E\left\|S(t-\tau )(f(u_1)-f(u_2))\right\|^2d\tau\\
\leq	& c^2t \int_0^t  \mathbb E\left\| u_1(\tau)-u_2(\tau)\right\|^2d\tau\\
\leq& c^2 t^2 \|u_1-u_2\|^2_c.
\end{align*}
By induction, it is easy to get
$$
\mathbb E \|P^nu_1(t)-P^nu_2(t)\|^2\leq \frac{c^{2n} t^{2n}}{(2n+1)!!} \|u_1-u_2\|^2_c.
$$
Since $(2n+1)!!$ is faster growth than $c^{2n} T^{2n}$ for $1\ll n$, there is a $\hat{n}$ such that
$\frac{c^{2\hat{n}}T^{2\hat{n}}}{(2\hat{n}+1)!!}:=\sigma^2<1$, it holds $ \|P^nu_1 -P^nu_2 \|_c\leq \sigma \|u_1-u_2\| _c. $ The contractility of $P^{\hat{n}}$ follows, and then  $P^{\hat{n}}$ has a unique fixed point $u_*$ belonging to $C(I,L^2(D,H))$ by using the Banach fixed point theorem. By virtue of $P^{\hat{n}}P=P^{\hat{n}+1}=PP^{\hat{n}}$, one see that
$P^{\hat{n}}(Pu_*)=P(P^{\hat{n}}u_*)=Pu_*$, which deduces that $Pu_*$ is a fixed point of $P^{\hat{n}}$, thus the uniqueness of fixed point implies that $Pu_*=u_*$, this shows that $u_*$ is also a fixed point of $P$, which is the solution to problem (\ref{P}). The proof is complete.
\end{proof}

\subsection{Spatial and temporal regularities}
We now show the spatial regularity of solution.
\begin{theorem}\label{thm3}
Let $u$ be the solution of problem (\ref{P}), $g\in \dot{H}^\sigma(D)$ and let $f$ satisfy the assumptions (\ref{f}). Let $s> \frac{1- H_1}{2} $, $2sH_2/\alpha+ H_1-1 >0$ and $2\sigma\in [0, \min\{2s+H_1-1,{2sH_2 }/{\alpha}+H_1-1\})$, then
$$
\mathbb E\|  u(t)\|^2_{\dot{H}^\sigma(D)}\leq c.
$$
\end{theorem}

\begin{proof}
	
From Theorem \ref{thm2}, we have
\begin{align*}
\mathbb E\|  u(t)\|^2_{\dot{H}^\sigma(D)}\leq& 3 \mathbb E \| S(t )g\|^2_{\dot{H}^\sigma(D)}+ 3\mathbb E \left\| \int_0^t  S(t-\tau )f(u)d\tau\right\|^2_{\dot{H}^\sigma(D)}\\
&+3\mathbb E \left\|\int_0^t \int_D   F(t-\tau,x,y)\gamma(\tau)\xi^{H_1,H_2} (dy,d\tau)\right\|^2_{\dot{H}^\sigma(D)}\\
\leq& 3Q_1+3Q_2+3Q_3.
\end{align*}
To estimate $Q_1$, we note that  $|u_k(t )|\leq 1 $ and then
\begin{align*}
Q_1 =\mathbb E\left(  \sum_{k=1}^\infty \rho_k^{2 \sigma}|u_k(t)|^2|g_k|^2\right)\leq  \mathbb E\|g\|^2_{\dot{H}^\sigma(D)}.
\end{align*}

By (\ref{inter2}) and the assumption of $f$, the Cauchy-Schwarz inequality implies
\begin{align*}
Q_2 \leq&  \mathbb E \left\| \int_{\Gamma_{\kappa,\pi-\vartheta}} e^{zt} \rho_k^\sigma \hat{u}_k(z) \hat{f}(u)dz\right\|^2
\leq	 c \left (\int_0^t (t-\tau)^{-\frac{\alpha\sigma}{s}}\mathbb E  \| f(u(\tau))\|d\tau  \right)^2\\
\leq	&c\, t^{\epsilon}  \int_0^t (t-\tau)^{-\frac{2\alpha\sigma}{s}+1-\epsilon}(1+\mathbb E  \| u(\tau)\|^2)d\tau  \\
\leq &c \left(1+  \int_0^t (t-\tau)^{-\frac{2\alpha\sigma}{s}+1-\epsilon} \mathbb E  \| u(\tau)\|^2 d\tau\right),
\end{align*}
where $\hat{f}$ is the Laplace transform of $f$, we also require $-2\alpha\sigma/s+1>-1$ for $\sigma\leq s$, that is $\sigma<s/\alpha$, $\epsilon\in(0,2(1-\alpha\sigma)/s )$.

To estimate $Q_3$, we first have for any $\eta\in[0,1]$
\begin{align*}
Q_3 \leq&  c\sum_{k=1}^\infty \int_0^t \left|\partial_\tau^{\frac{1-2H_2}{2} } \rho_k^\sigma u_k(t-\tau)\gamma(\tau)\right|^2 d\tau\|e_k(y)\|_{H_0^{\frac{1-2H_1}{2}}(D)}^2\|e_k(x)\|^2\\
\leq &  c\sum_{k=1}^\infty \int_0^t \left|\partial_\tau^{\frac{1-2H_2}{2} } \rho_k^{\sigma+\frac{1-2H_1}{4}} u_k(t-\tau)\gamma(\tau)\right|^2 d\tau\\
\leq &  c\sum_{k=1}^\infty\rho_k^{-2s\eta+2\sigma+\frac{1-2H_1}{2}} \int_0^t \left|\partial_\tau^{\frac{1-2H_2}{2} } \rho_k^{s\eta}   u_k(t-\tau)\gamma(\tau)\right|^2 d\tau.
\end{align*}
According to Lemma \ref{deri} by replacing $\varsigma=\frac{1-2H_2}{2}$ for $\eta<H_2/\alpha$, we have
\begin{align*}
 \sum_{k=1}^\infty\rho_k^{-2s\eta+2\sigma+\frac{1-2H_1}{2}} <\infty ,
\end{align*}
for $\eta =(2\sigma+\frac{1-2H_1}{2}+\epsilon)/2s\in [0,1]$ with $\epsilon>1/2$, where we require $2\sigma+\frac{1-2H_1}{2}+\epsilon<2sH_2/\alpha$ and $2\sigma+\frac{1-2H_1}{2}+\epsilon\leq 2s$, that is $2\sigma<\min\{2s+H_1-1,\frac{2sH_2 }{\alpha}+H_1-1\}=:\sigma_M$, $\epsilon\in(1/2,\sigma_M-2\sigma+1/2)$. This implies that $Q_3\leq c$. Together the estimates $Q_1,Q_2,Q_3$ and the Gronwall-Henry inequality,  the desired result follows from the embedding $\dot{H}^\sigma(D)\subset H$. The proof is complete.

\end{proof}
\begin{remark}
The requirement $s> \frac{1- H_1}{2} $ implies that $s>1/4$ for $0<H_1\leq 1/2$.
\end{remark}

\begin{remark}
	Under the assumptions of Theorem \ref{thm3} and by replacing the condition on \( g \) with \( g \in H \), the estimate from Lemma \ref{lemma-3.1} can imply that
	\( Q_1 \leq c t^{-2\alpha} \mathbb{E} \|g\|^2. \)
	This estimate further indicates that there exists a power decay property for nonsmooth data, satisfying
	\( \| u(t) \|_{\dot{H}^\sigma(D)} \leq c t^{-\alpha}, \)
	for all \( t > 0 \).
	
\end{remark}

We next prove the temporal H\"{o}lder regularity of solution.
\begin{theorem}\label{thm3-4}
	Let $u$ be the solution to problem (\ref{P}), $g\in H$ and let $f$ satisfy the assumptions (\ref{f}). Let $s>\frac{1-H_1}{2}$, $2sH_2/\alpha+ H_1-1 >0$ and $ \delta\in [0,     H_2 +(H_1-1)\alpha/(2s))$, then
	$$
	\mathbb E\left \|\frac{ u(t)-u(t-h)}{h^\delta}\right\|^2\leq c(t-h)^{-2\delta}\mathbb E\|g\|^2+c.
	$$
\end{theorem}

\begin{proof}
From Theorem \ref{thm2}, we have
\begin{align*}
	\mathbb E \left \|\frac{ u(t)-u(t-h)}{h^\delta}\right\|^2\leq& 3 \mathbb E \left \|\frac{1}{h^\delta}( S(t )g-S(t -h)g)\right\|^2\\
	& + 3\mathbb E \left\| \frac{1}{h^\delta}\bigg(\int_0^t  S(t-\tau )f(u)d\tau-\int_0^{t-h}  S(t-h-\tau )f(u)d\tau\bigg)\right\|^2\\
	&+3\mathbb E \bigg\|\frac{1}{h^\delta}\bigg(\int_0^t \int_D  F(t-\tau,x,y)\gamma(\tau)\xi^{H_1,H_2} (dy,d\tau)\\
	&-\int_0^{t-h} \int_D  F(t-h-\tau,x,y)\gamma(\tau)\xi^{H_1,H_2} (dy,d\tau)\bigg)\bigg\|^2\\
	\leq& 3J_1+3J_2+3J_3.
\end{align*}
By Lemma \ref{lemma-3.1} (iv), it follows by Cauchy-Schwarz inequality that
\begin{align*}
J_1 =&\frac{1}{h^{2\delta}}\mathbb E \left (\sum_{k=1}^\infty |u_k(t)-u_k(t-h)|^2|g_k|^2 \right)
\leq  c (t-h)^{-2\delta}  \mathbb E \|g\|^2,\\
J_2\leq &\frac{1}{h^{2\delta}}\mathbb E \left\|  \int_0^{t-h}  (S(t-\tau )-S(t-h-\tau ))f(u)d\tau+\int_{t-h}^t  S(t -\tau )f(u)d\tau \right\|^2\\
\leq & 2c^2  (t-h )^{1-\delta} \int_0^{t-h}   (t-h-\tau )^{-\delta}\mathbb E \|f(u)\|^2d\tau   +\frac{2c^2h}{h^{2\delta}} \int_{t-h}^t  \mathbb E \|f(u)\|^2d\tau  \\
\leq & c   \int_0^{t-h}   (t-h-\tau )^{-\delta}(1+\mathbb E \| u(\tau)\|^2)d\tau  +c h^{1-2\delta}  \int_{t-h}^t  (1+\mathbb E \| u(\tau)\|^2)d\tau  \leq c,
\end{align*}
where $\mathbb E \| u(t)\|\leq c$ for all $t\in I$ from Theorem \ref{thm3}.
As for $J_3$, splitting the proof into two parts
 \begin{align*}
J_3 \leq&  2\mathbb E \bigg\|\frac{1}{h^\delta} \int_0^{t-h}  \int_D  (F(t-\tau,x,y)-F(t-h-\tau,x,y))\gamma(\tau)\xi^{H_1,H_2} (dy,d\tau)\bigg\|^2\\
 	&+2\mathbb E \bigg\|\frac{1}{h^\delta}\int_ {t-h}^t \int_D  F(t -\tau,x,y)\gamma(\tau)\xi^{H_1,H_2} (dy,d\tau) \bigg\|^2
 	\leq  2J_{31}+2J_{32}.
 \end{align*}
By applying the similar proof in Theorem \ref{thm3} and (\ref{inter2}) yield
\begin{align*}
J_{31}\leq&   \mathbb E \bigg\|\frac{1}{h^\delta} \int_0^{t-h}  \int_D  (F(t-\tau,x,y)-F(t-h-\tau,x,y))\gamma(\tau)\xi^{H_1,H_2} (dy,d\tau)\bigg\|^2\\
	\leq &  \frac{c}{h^{2\delta}}\sum_{k=1}^\infty \int_0^{t-h} \left|\partial_\tau^{\frac{1-2H_2}{2} }   (u_k(t-\tau)-u_k(t-h-\tau))\gamma(\tau)\right|^2 d\tau\|e_k(y)\|_{H_0^{\frac{1-2H_1}{2}}(D)}^2\|e_k(x)\|^2\\
	\leq &  \frac{c}{h^{2\delta}}\sum_{k=1}^\infty \rho_k^{ \frac{1-2H_1}{2}}\int_0^{t-h} \left|\partial_\tau^{\frac{1-2H_2}{2} }  (u_k(t-\tau)-u_k(t-h-\tau))\gamma(\tau)\right|^2 d\tau\\
	\leq &  \frac{c}{h^{2\delta}}\sum_{k=1}^\infty\rho_k^{-\frac{1}{2}-2\epsilon} \int_0^{t-h} \left|\partial_\tau^{\frac{1-2H_2}{2} } \rho_k^{ \frac{1- H_1}{2}+\epsilon} (u_k(t-\tau)-u_k(t-h-\tau))\gamma(\tau)\right|^2 d\tau.
\end{align*}
By using the similar proof in Lemma \ref{lemma-3.1}, we have $|u''_k(t)|\leq c\rho_k^{s\eta} t^{-\alpha\eta-2}$, as the analogous proof in Lemma \ref{deri} with $\frac{1-2H_2}{2}+\delta+\alpha\eta<1/2$, the following estimate holds
$$
\int_0^{t-h} \left|\partial_\tau^{\frac{1-2H_2}{2} } \rho_k^{ \frac{1- H_1}{2}+\epsilon} (u_k(t-\tau)-u_k(t-h-\tau))\gamma(\tau)\right|^2 d\tau\leq ch^{2\delta}  ,
$$
for $\eta :=(\frac{1- H_1}{2}+\epsilon)/s $ and we require $0<2sH_2/\alpha+H_1-1$ and $\frac{1- H_1}{2} < s$ with small $\epsilon\in (0,sH_2/\alpha+(H_1-1)/2))\cap [0,s)$. Thus, we have $J_{31}\leq c$ and $\delta+\alpha\eta<H_2$.
Similarly, it follows from the variable substitutions that
\begin{align*}
\int_{t-h}^t \bigg|\partial_\tau^{\frac{1-2H_2}{2} }  &\rho_k^{s\eta } u_k(t-\tau) \gamma(\tau)\bigg|^2 d\tau\leq c \int_{t-h}^t  \left(\int_0^\tau k_{H_2+\frac{1 }{2} } (\tau-\upsilon)  (t-\upsilon)^{-\alpha\eta  -1}    d\upsilon \right)^2 d\tau\\
=&c \int_{0}^h  \left(\int_\tau^{t} k_{H_2+\frac{1 }{2} } ( \upsilon-\tau)   \upsilon ^{-\alpha\eta  -1}    \right)^2d\upsilon d\tau \\
=&c \int_{0}^h \left( \int_1^{t/\tau} k_{H_2+\frac{1 }{2} } ( \upsilon-1)   \upsilon ^{-\alpha\eta  -1}\right)^2 \tau ^{2(H_2-\alpha\eta  )-1}d\upsilon d\tau \\
=&c \int_{0}^h  \left(\int_0^{t/\tau-1} k_{H_2+\frac{1 }{2} } ( \upsilon )   (\upsilon+1) ^{-\alpha\eta  -1}d\upsilon\right)^2 \tau ^{2(H_2-\alpha\eta  )-1}d\tau,
\end{align*}
By using the identity
$$
\int_0^\infty \upsilon^{\tilde{p}-1}(1+\upsilon)^{-(\tilde{p}+\tilde{q})}d\upsilon=\frac{\Gamma(\tilde{p})\Gamma(\tilde{q})}{\Gamma(\tilde{p}+\tilde{q}
)},~~\tilde{p},\tilde{q}>0,
$$
 with replaced by $\tilde{p}= H_2+\frac{1}{2}$, $\tilde{q}= \alpha\eta+\frac{1}{2}-H_2$, we thus get
\begin{align*}
J_{32} \leq&    \frac{c}{h^{2\delta}}\sum_{k=1}^\infty \int_{t-h}^t  \left|\partial_\tau^{\frac{1-2H_2}{2} }    \rho_k^{ \frac{1-2H_1}{4}} u_k(t-\tau) \gamma(\tau)\right|^2 d\tau \\
	\leq &  \frac{c}{h^{2\delta}}\sum_{k=1}^\infty\rho_k^{-\frac{1}{2}-2\epsilon} \int_{t-h}^t \left|\partial_\tau^{\frac{1-2H_2}{2} } \rho_k^{ \frac{1- H_1}{2}+\epsilon}  u_k(t-\tau) \gamma(\tau)\right|^2 d\tau
	\leq   c h^{2(H_2-\alpha\eta-\delta) }.
\end{align*}
Together above estimates, the desired results are obtained.
\end{proof}

\section{Wong-Zakai approximation and error estimate}\label{sect4}

We next consider the case of initial condition $g\equiv 0$ in problem (1.1).
Let $\tau=T/M$, $I_i=(t_i,t_{i+1}]$  with $t_i=i\tau$ for $i=0,1,\cdots,M$. We denote the Wong-Zakai approximation by using the spectral bases $\{e_k\}_{k=1}^N$ and piecewise constant function to approximate $\xi^{H_1,H_2}(x,t)$ as
$$
\xi_W^{H_1,H_2}(x,t)=\sum_{i=1}^{M}\sum_{j=1}^{N} \left(
\frac{1}{\tau }\int_{I_i}\int_{D }e_k(y) \xi^{H_1,H_2}(dy,d\upsilon)\right)e_k(x) \chi_{I_i}(t),
$$
where $\chi_{I_i }( t)$ is the characteristic function on $I_i=(t_{i-1},t_i] $. Let $u_W(x,t) $ be the solution of the regularized equation
\begin{equation}\label{regu}
	\left\{
	\begin{aligned}
  \partial_tu_W &+\lambda \partial^{1-\beta}_{t} u_W +\mu \partial^{1-\alpha} _{t}  A^s u_W =f(u_W)+\gamma(t) \xi_W^{H_1,H_{2}}(x,t) ,\qquad {\rm in}\ D\times (0,T],\\
	u_W &=0,\qquad  {\rm on}\  \partial D \times (0,T],\\
	u_W &=0 ,\qquad {\rm in}\   D \times\{t=0\},
	\end{aligned}
	\right.
\end{equation}
By a direct calculation, we have
$$
u_W(x,t)=\int_0^t S(t-\upsilon)f(u_W)d\upsilon+\int_0^t\int_D F_W(t,\upsilon,x,y)\gamma(\upsilon) \xi_W^{H_1,H_{2}}(dy,d\upsilon),
$$
where
$$
F_{W,N}(t,\upsilon,x,y)=\sum_{k=1}^N u_{W,k}(t,\upsilon)e_k(x)e_{ k}(y),
$$
and
$$
u_{W,k}(t,\upsilon)=\frac{1}{\tau} \sum_{i=1}^{M } \chi_{I_i}(\upsilon) \int_{I_i} u_k(t-r)\chi_{(0,t)}(t-r)dr,
$$

\begin{theorem}\label{thm4}
	Let $u_W$ be the solution to problem (\ref{regu}), and $f$ satisfies (\ref{f}).  Let $s>\frac{1-H_1}{2}$, $2sH_2/\alpha+ H_1-1 >0$ and $2\sigma\in [0,\min\{2s+H_1-1, {2H_2}/{\alpha}+H_1-1\})$, then
	$$
	\mathbb E\| u_W(t)\|^2_{\dot{H}^\sigma(D)}\leq c.
	$$
\end{theorem}

\begin{proof}
	From the similar proof in Theorem \ref{thm3}, we have
\begin{align*}
	\mathbb E\| u_W(t)\|^2_{\dot{H}^\sigma(D)}\leq & 2	\mathbb E\left \|\int_0^t  S(t-\upsilon)f(u_W)d\upsilon\right\|^2_{\dot{H}^\sigma(D)}\\ &+ 2	\mathbb E\left \|\int_0^t\int_D F_{W,N}(t,\upsilon,x,y)\gamma(\upsilon) \xi ^{H_1,H_{2}}(dy,d\upsilon)\right\|^2_{\dot{H}^\sigma(D)}\\
		\leq& c\left( 1+\int_0^t (t-\upsilon)^{-2\alpha\sigma/s+1-\epsilon} \mathbb E  \| u_W(\upsilon)\|^2 d\upsilon\right)+Q.
\end{align*}	
To estimate $Q$, by using the standard approximation theory, we note that
\begin{align*}
 \left\|\partial_r^{\frac{1-2H_2}{2} } u_{W,k}(t,r)\gamma(r)\right\|^2_{L^2(0,t)}\leq &\left\|\partial_r^{\frac{1-2H_2}{2} } (u_{W,k}(t,r)\gamma(r)-u_k(t-r)\gamma(r)\chi_{(0,t)}(r))\right\|^2_{L^2(0,t)}\\
 &+ c \left\|\partial_r^{\frac{1-2H_2}{2} }   u_k(t-r)\gamma(r)\chi_{(0,t)}(r) \right\|^2_{L^2(0,t)}\\
 \leq & \left\|\partial_r^{\frac{1-2H_2}{2} }   u_k(t-r)\gamma(r)  \right\|^2_{L^2(0,t)}.
\end{align*}	
Consequently, the previous proof in Theorem \ref{thm3} shows that
$Q\leq c$. And then the Gronwall-Henry inequality implies the desired estimate.

\end{proof}

In what follows,  an error estimate for the Wong-Zakai
approximation of $\xi_W$ are obtained, and the convergence rate of the solution $u_W$ of (\ref{regu}) to the solution $u$ of (\ref{P}) also established in term of $\tau$.
\begin{theorem}\label{thm5}
Let $u$ and $u_W$ be the solution to problem (\ref{P}) and (\ref{regu}), respectively. Let $f$ satisfy (\ref{f}), and let $s>\frac{1-H_1}{2}$, $2sH_2/\alpha+H_1-1 >0$ and $2\sigma\in [0,\min\{2sH_2/\alpha+H_1-1,2s+H_1-1\})$,  then
$$
\mathbb E\| u(t)- u_W(t)\|^2\leq c((N+1)^{-4\sigma}+N^{1-2H_1}\tau^{2H_2-\frac{\alpha}{2s}-\epsilon }),
$$
for some small $\epsilon\in(0,2H_2-\frac{\alpha}{2s})$.	
\end{theorem}

\begin{proof}
We first have
\begin{align*}
	\mathbb E\| u(t)&- u_W(t)\|^2\leq 2	\mathbb E\left \|\int_0^t   S(t-\upsilon)(f(u)-f(u_W)d\upsilon\right\|^2\\ &+ 2	\mathbb E\left \|\int_0^t\int_D (F(t-\upsilon,x,y)\chi_{(0,t)}(t-\upsilon)-F_{W,N}(t,\upsilon,x,y))\gamma(\upsilon) \xi ^{H_1,H_{2}}(dy,d\upsilon)\right\|^2\\ \leq & R_1+R_2 .
\end{align*}	
The assumptions of $f$ and $\| S(t)v\|\leq \|v\|$ imply
\begin{align*}
R_1 \leq   t\int_0^t 	\mathbb E\left \|f(u)-f(u_W)\right\|^2 d\upsilon  \leq   c^2t \int_0^t 	\mathbb E\left \|u(\upsilon)-u_W(\upsilon)\right\|^2 d\upsilon.
\end{align*}	
To estimate $R_2$, we split it into two parts:
\begin{align*}
 R_2\leq &   c\mathbb E\left \|\int_0^t\int_D \sum_{k=1}^N (u_k(t-\upsilon)\chi_{(0,t)}(t-\upsilon)-u_{W,k}(t,\upsilon))\gamma(\upsilon) e_k(x)e_k(y) \xi ^{H_1,H_{2}}(dy,d\upsilon)\right\|^2\\
 &+c\mathbb E\left \|\int_0^t\int_D \sum_{k=N+1}^\infty u_k(t-\upsilon)\gamma(\upsilon) e_k(x)e_k(y) \xi ^{H_1,H_{2}}(dy,d\upsilon)\right\|^2\leq R_{21}+R_{22}.
\end{align*}
%
Since from Lemma \ref{xi-H} again,
it yields
\begin{align*}
R_{22} \leq& c   \sum_{k=N+1}^\infty \left\|\partial_\upsilon^{\frac{1-2H_2}{2}} u_k(t-\upsilon) \gamma(\upsilon)\right\|^2_{L^2(0,t)} \|e_k(y)\|_{H_0^{\frac{1-2H_1}{2}}(D)}^2 \\
\leq & c  \rho_{N+1}^{ -2\sigma } \sum_{k=N+1}^\infty\rho_k^{-\frac{1 }{2}-\epsilon}\left\|\partial_\upsilon^{ \frac{1- 2H_2 }{2}} \rho_k^{\frac{1+2\epsilon+4\sigma+ 1-2H_1 }{4}} u_k(t-\upsilon) \gamma(\upsilon)\right\|^2_{L^2(0,t)},
\end{align*}
where we need $2\sigma<\min\{2sH_2/\alpha+H_1-1,2s+H_1-1\}$ and choosing small $\epsilon>0$, in view of $\rho_k\simeq k^2$. To estimate $R_{21}$, by using the inverse estimate  and the projection theorem, by $\gamma\in C^1(I)$ and Remark \ref{Re-2} we have
\begin{align*}
R_{21}\leq &c\sum_{k=1}^N \left\|\partial_\upsilon^{\frac{1-2H_2}{2}} (u_k(t-\upsilon)\chi_{(0,t)}(t-\upsilon)-u_{W,k}(t,\upsilon))\gamma(\upsilon)\right\|^2_{L^2(0,t)} \|e_k(y)\|_{H_0^{\frac{1-2H_1}{2}}(D)}^2 \\
\leq &  cN^{1-2H_1}\sum_{k=1}^N \rho_k^{-\frac{1}{2}- 2\epsilon } \left\|\partial_\upsilon^{\frac{1-2H_2}{2}}\rho_k^{ \frac{1+4\epsilon }{4} }(u_k(t-\upsilon)-u_{W,k}(t,\upsilon) ) \gamma(\upsilon)\right\|^2_{L^2(0,t)} \\
\leq &  cN^{1-2H_1}\tau^{2H_2-\frac{\alpha}{2s}-\epsilon}\sum_{k=1}^\infty \rho_k^{-\frac{1-4\epsilon}{2} } \left\|\partial_\upsilon^{\frac{1-\epsilon  }{2}-\frac{\alpha}{4s}}\rho_k^{ \frac{1+4\epsilon }{4} } u_k(t- \upsilon)  \gamma(\upsilon) \right\|^2_{L^2(0,t)}\\
\leq & cN^{1-2H_1}\tau^{2H_2-\frac{\alpha}{2s}-\epsilon },
 \end{align*}	
for some small $\epsilon>0$.

\end{proof}

\begin{remark}
We note that, since $H_1\leq 1/2$, the total error $N^{1-2H_1}\tau^{2H_2-\frac{\alpha}{2s}-\epsilon }$ admits a certain connection that makes the error in Theorem \ref{thm5} very small when $N$ is large, one case if $\tau\leq 1/N$, there should be required  $2H_1+2H_2>\frac{\alpha}{2s}+1$. Moreover, the requirement $\alpha<4sH_2$ for $H_2\leq 1/2$ is also needed.
\end{remark}

\section{Spatial/temporal discretizations and error analysis}\label{sect5}

In this section, we delve into the spatial and temporal discretizations, as well as the associated error analysis.
\subsection{Spatial discretization}
To discretize the fractional Laplacian, we use a spectral
Galerkin method,
now introducing a finite-dimensional subspace of $H$ by $H_N=\{e_k\}_{k=1}^N$ for the spectral cutoff $N\in\mathbb N$ in Wong-Zakai approximation  and the projection operator $P_N$ by
$$
(P_N\varphi,v_N)=(\varphi,v_N),~~\forall v_N\in H_N.
$$
By a simple calculation, it deduces that
$$
P_Nu=\sum_{k=1}^N (u,e_k)e_k,~~~\forall u\in H.
$$
Upon introducing the discrete fractional Laplacian $A_N^s:H_N\to H_N$ by
$$
(A_N^s u_N,v_N)=(A^s u,v_N),~~~\forall u_N,v_N\in H_N,
$$
which leads to
$$
A_N^su_N=A^s_NP_Nu_N=P_NA_N^su_N=\sum_{k=1}^N\rho_k^s (u_N,e_k)e_k,~~~\forall u\in H.
$$
We rewrite the spatially semi-discrete problem (\ref{P}) with $g\equiv 0$ based on a spatial Wong-Zatai approximation as to find $u\in H_N$ such that
\begin{equation}\label{approximation}
\partial_tu_N   +\lambda \partial^{1-\beta}_{t} u_N +\mu \partial^{1-\alpha} _{t}  A^s_N u_N =P_Nf(u_N)+\gamma(t) P_N\xi_{WS}^{H_1,H_{2}}(x,t) ,
\end{equation}
with $u_N(0)=0$, where
$$
\xi_{WS}^{H_1,H_2}(x,t)=\sum_{k=1}^{N} \int_{D }e_k(y) \xi^{H_1,H_2}(dy,dt)e_k(x).
$$
Accordingly, we introduce an operator $S_N(t)$ by its Laplace transform
$$
\hat{S}_N(z)=\mu^{-1}z^{\alpha-1}(h(z)+A_N^s)^{-1},~~h(z)=\mu^{-1}z^{\alpha}(1+\lambda z^{-\beta}),
$$
with its inverse Laplace transform showing
\begin{equation}\label{SN}
S_N(t)=\frac{1}{2\pi i}\int_{\Gamma_{\kappa,\pi-\vartheta}} e^{zt} \mu^{-1}z^{\alpha-1}(h(z)+A_N^s)^{-1} dz.
\end{equation}
Furthermore,
the discrete norm $\|\cdot\|_{*H_0^s}$ on the space $H_N$ for any $s\geq0$
$$
\|\varphi\|_{*H_0^s(D)}=\left(\sum_{k=1}^N \rho_{k }^s (\varphi,e_{k })^2\right)^{1/2},~~\forall \varphi \in H_N.
$$
\begin{remark}
Similarly to that of
Theorem \ref{Thm1},  let $S_N(t)$ be defined in (\ref{SN}) and $v_N\in H_N$, then the stability of the operator $S_N(t)$ is given as follows
	$$
	\|S^{(m)}_N(t)v_N\|_{*\dot{H}^q(D)}\leq ct^{-m-\alpha(q-p)/(2s)}\|v_N\|_{*\dot{H}^p(D)},
	$$
	for $0\leq p\leq q\leq 2s$ whenever $m=0$ or $0\leq p,q\leq 2s$ whenever $m\geq1$.
	
\end{remark}

By the definitions of $P_N$ and $A_N^s$, the solution $u_N$ of  regularized problem is given by
$$
u_N(t)=\int_0^t S_N(t-\upsilon)P_Nf(u_N)d\upsilon+\int_0^t S_N(t-\upsilon)\gamma(\upsilon)P_N  \xi_{WS}^{H_1,H_{2}}( d\upsilon).
$$
Let
$$
F_{W,P}(t,x,y)=\sum_{k=1}^{N} F_{W,k}(t,x,y),
$$
for $F_{W, k}(t,x,y)=S_{ N}(t) P_Ne_k(x)e_{ k}(y)$. Then, the solution $u_N$ can be written as
\begin{equation}\label{120}
\begin{aligned}
u_N(t)=&\int_0^t S_N(t-\upsilon)P_Nf(u_N)d\upsilon +\int_0^t\int_D  F_{W,P}(t-\upsilon,x,y)\gamma(\upsilon)  \xi^{H_1,H_{2}}(dy, d\upsilon).
\end{aligned}
\end{equation}

By virtue of the stability of $P_N$, $\|P_Nu\|\leq \|u\|$, being similar to the proofs of Theorems \ref{thm2} and \ref{thm3}, the following results are easy to obtained.
\begin{theorem}\label{thm6}
	Let $u_N$ be the solution to  problem (\ref{approximation}), and $f$ satisfy the assumptions (\ref{f}). Let $s>\frac{1-H_1}{2}$,  $2sH_2/\alpha+ H_1-1 >0$ and $2\sigma\in [0,\min\{2s+H_1-1,{2sH_2 }/{\alpha}+H_1-1\})$, then
	$$
	\mathbb E\|  u_N(t)\|^2_{*\dot{H}^\sigma(D)}\leq c.
	$$
\end{theorem}

\begin{theorem}\label{thm7}
	Let $u_N$ be the solution to  problem (\ref{approximation}),  and $f$ satisfy the assumptions (\ref{f}). Let $s>\frac{1-H_1}{2}$, $2sH_2/\alpha+ H_1-1 >0$ and
$ \delta\in [0,     H_2 +(H_1-1)\alpha/(2s))$, then
	$$
	\mathbb E\left \|\frac{ u_N(t)-u_N(t-h)}{h^\delta}\right\|^2\leq c.
	$$
\end{theorem}
The next lemma shows an error estimate between $(h(z) + A^s)^{-1}v$ and its discrete form $(h(z) + A_N^s)^{-1}P_Nv$.
\begin{lemma}\label{S-SN}
	Let $v\in H$, for any $\eta\in[0,1]$, $v\in H$,  then
	$$
\|S(t)v-S_N(t)P_Nv\|\leq c(N+1)^{-2s\eta}t^{-\alpha\eta}\|v\|.
	$$
\end{lemma}
\begin{proof}
	By the inverse Laplace transform of $S(t)$ in (\ref{St}) and $S_N(t)$ in (\ref{SN}), for any given $v\in H$, we have
\begin{align*}
\|(S(t)-S_N(t)P_N)v\|^2\leq& c^2\sum_{k=N+1}^\infty  |u_k(t)|^2|v_k|^2.
\end{align*}
By using the estimate in Lemma \ref{lemma-3.1}, by interpolation we have
$$
\|(S(t)-S_N(t)P_N)v\|^2\leq c^2 \sum_{k=N+1}^\infty \rho_k^{-2s\eta} t^{-2\alpha\eta}|v_k|^2,
$$
this leads that
\begin{align*}
	\|S(t)-S_N(t)v\|^2\leq& c\sup_{k\geq N+1}\rho_k^{-2s\eta}  t^{-2\alpha\eta}\|v\|^2
\end{align*}
Hence the desired estimate is shown.

\end{proof}

We now state the error estimate.
\begin{theorem}\label{thm8}
	Let $u$ and $u_N$ be the solutions of problem (\ref{P}) and (\ref{approximation}), respectively. Let $2\sigma<\min\{2sH_2/\alpha+H_1-1,2s+H_1-1\}$,   there holds
	$$
	\mathbb E\|u(t)-u_N(t)\|^2\leq  c  (N+1)^{-4\sigma} .
	$$
\end{theorem}

\begin{proof}
To estimate the error $e(t) := u(t)- u_N(t)$, we have
\begin{align*}
\mathbb E\|e(t)\|^2\leq & 2\mathbb E\left\| \int_0^t S(t-\upsilon)f(u)d\upsilon-\int_0^t S_N(t-\upsilon)P_Nf(u_N)d\upsilon\right\|^2\\
&+2\mathbb E\left\|\int_0^t\int_D (F(t-\upsilon,x,y)-F_{W,P}(t-\upsilon,x,y))\gamma(\upsilon)\xi^{H_1,H_2}(dy,d\upsilon)\right\|^2\\ \leq& 2\tilde{Q}_1+2\tilde{Q}_2.
\end{align*}
Accordingly Lemma \ref{S-SN}, to estimate $\tilde{Q}_1$,  we have
\begin{align*}
\tilde{Q}_1\leq  & 2 \mathbb E\left\| \int_0^t( S(t-\upsilon)-S_N(t-\upsilon)P_N)f(u_N)d\upsilon\right\|^2+2\mathbb E\left\|\int_0^t S (t-\upsilon)( f(u)- f(u_N))d\upsilon\right\|^2\\
\leq & 2c^2(N+1)^{-4s\eta}  \mathbb E\left(\int_0^t(t-\upsilon)^{- \alpha\eta} \left\|f(u)\right\| d\upsilon\right)^2+2c^2  \mathbb E\left(\int_0^t \mathbb E\left\|f(u)- f(u_N)\right\| d\upsilon\right)^2\\
\leq &2c^2(N+1)^{-4s\eta}t^{\alpha\eta} \int_0^t(t-\upsilon)^{-\alpha\eta}(1+\mathbb E\left\|u(\upsilon)\right\|^2)d\upsilon+2c^2 t\int_0^t \mathbb E\left\|e(\upsilon)\right\|^2d\upsilon.
\end{align*}
To estimate $\tilde{Q}_2$, we have
\begin{align*}
\tilde{Q}_2\leq   &  c\sum_{k=N+1}^\infty  \int_0^t \left|\partial_\tau^{\frac{1-2H_2}{2} }  u_k(t-\tau)\gamma(\tau)\right|^2 d\tau\|e_{k}(y)\|_{H_0^{\frac{1-2H_1}{2}}(D)}^2\|e_k(x)\|^2\\
\leq & c  \rho_{N+1}^{ -2\sigma } \sum_{k=N+1}^\infty\rho_k^{-\frac{1 }{2}-\epsilon}\left\|\partial_\upsilon^{ \frac{1- 2H_2 }{2}} \rho_k^{\frac{1+2\epsilon+4\sigma+ 1-2H_1 }{4}} u_k(t-\upsilon) \gamma(\upsilon)\right\|^2_{L^2(0,t)}\\
\leq&   c(N+1)^{-4\sigma} ,
\end{align*}
which is from $2\sigma<\min\{2sH_2/\alpha+H_1-1,2s+H_1-1\}$ and  small $\epsilon>0$. We thus get the main result for choosing $\eta=\sigma/s\in(0,1)$.
\end{proof}

Combining the conclusions of Theorem \ref{thm5} and Theorem \ref{thm8}, we have the following result.
\begin{theorem}\label{thm5-8}
Let $u_W$ and $u_N$ be the solution to problem (\ref{regu}) and (\ref{approximation}), respectively. Let  $s>\frac{1-H_1}{2}$,  $2sH_2/\alpha+H_1-1 >0$ and $2\sigma\in [0,\min\{2sH_2/\alpha+H_1-1,2s+H_1-1\})$,    then
$$
\mathbb E\| u_W(t)-u_N(t)\|^2\leq c((N+1)^{-4\sigma}+N^{1-2H_1}\tau^{2H_2-\frac{\alpha}{2s}-\epsilon }),
$$
for some small $\epsilon\in(0,1)$.	
\end{theorem}

\subsection{Temporal discretization}
In this subsection, let $\tau=T/M$ be the discretization mesh in Wong-Zakai approximation, we use the backward Euler (BE) convolution
quadrature to discretize the Riemann-Liouville fractional derivative.
%
$$
\partial_t^{1-\alpha } v(t_n)\approx\sum_{i=0}^{n-1} d^{(1-\alpha)}_iv(t_{n-i}),
$$
where
$$
\sum_{i=0}^{\infty} d_i^{(1-\alpha)}\zeta^i=(\delta_\tau(\zeta))^{1-\alpha},~~~\delta_\tau(\zeta)=\frac{1-\zeta}{\tau}.
$$
The fully discrete scheme of problem (\ref{regu}) can be written as
\begin{equation}\label{approximation2}
\frac{u_N^n-u_N^{n-1}}{\tau}+\lambda\sum_{i=0}^{n-1} d^{(1-\beta)}_i  u^{n-i}_N+\mu\sum_{i=0}^{n-1} d^{(1-\alpha)}_iA_N^s u^{n-i}_N= P_Nf(u_N^{n-1})+ P_N\xi^{H_1,H_2}_{W,n},
\end{equation}
where $
\xi^{H_1,H_2}_{W,n}=\gamma_n \xi^{H_1,H_2}_{ W}(t_n)$, $\gamma_n=\gamma(t_n)$.
Let $\overline{\mathscr F}(t)$ be defined by
$$
\overline{\mathscr F}(t):=\left\{
\begin{aligned}
&0, && t=t_0,\\
&f(u_N^{j-1} ), &&t\in (t_j,t_{j+1}],
\end{aligned}
\right.
$$
 and $\mathscr F(t):=f(u_N (t))$. Let $\mathscr F_N:=P_N\mathscr F$ and $\overline{\mathscr F}_N:=P_N\overline{\mathscr F}$. Let us introduce the operator $\overline{\partial}_\tau$ as
$$
\overline{\partial}_\tau W(t):=\left\{
\begin{aligned}
&0, && t=t_0,\\
&\frac{W(t_j)-W(t_{j-1})}{\tau}, &&t\in (t_j,t_{j+1}],\\
&0, &&t>T,
\end{aligned}
\right.
$$
Similarly, we denote $\overline{\partial}_\tau W_j(t)=0$ for $t=t_0$ and $t>t_N$, $\overline{\partial}_\tau W_j(t)=({W_j(t_n)-W_j(t_{n-1})})/{\tau}$, for $t\in (t_{n-1},t_{n}]$, $n=1,2,\cdot,N$,
it follows that the generating function
$\widetilde{\overline{\partial}_\tau u}(\cdot,\zeta)=\sum_{n=0}^\infty \overline{\partial}_\tau u(\cdot,t_n)\zeta^n$
is well-defined, see e.g. \cite{G-M}, which possesses the identity
$$
\widetilde{ \partial _\tau W}(\cdot,e^{-z\tau})=\frac{z}{e^{z\tau}-1}\widehat{\overline{\partial}_\tau W}(\cdot,z),
$$
where $\widehat{\cdot}$ denotes the Laplace transform and $\widetilde{v}(\zeta)=\sum_{n=0}^\infty v_n\zeta^n$, $\zeta\in\mathbb{D}$ the unit disk. Hence, by a transformation, we have
$$
\sum_{n=0}^{\infty}\overline{\mathscr F}_N(t_n)e^{-zt_n}=\frac{z}{e^{z\tau}-1}\widehat{\overline{\mathscr F}}_N(z),~~~\sum_{n=0}^{\infty} \xi_{W,n}^{H_1,H_2}e^{-zt_n}=\frac{z}{e^{z\tau}-1}\widehat{\gamma\xi _{W }^{H_1,H_2}}(z).
$$

Multiplying $\zeta^n$ on both sides and summing it from $n=0$ to infinity, we have
\begin{align*}
\delta_\tau(\zeta)\widetilde{u_N}(\zeta) &+\lambda (\delta_\tau(\zeta))^{1-\beta}\widetilde{u_N}(\zeta) +\mu(\delta_\tau(\zeta))^{1-\alpha}A_N^s \widetilde{u_N}(\zeta)
= \widetilde{\overline{\mathscr F}_N}(\zeta)  +    \widetilde{  P_N\xi^{H_1,H_2}_{W,n}}(\zeta),
\end{align*}
 by using the definition of $d^{(1-\varsigma)}_i$ for $\varsigma\in\{\beta,\alpha\}$.
By a calculation, we get
\begin{equation}\label{121}
\begin{aligned}
 \widetilde{u_N}(\zeta) =& \mu^{-1}(\delta_\tau(\zeta))^{\alpha-1}(h(\delta_\tau(\zeta)+A_N^s))^{-1} \widetilde{\overline{\mathscr F}_N}(\zeta)\\
 &  +\mu^{-1}(\delta_\tau(\zeta))^{\alpha-1}(h(\delta_\tau(\zeta)+A_N^s))^{-1}  \widetilde{ P_N\xi^{H_1,H_2}_{W,n}}(\zeta),
\end{aligned}
\end{equation}
For any $\theta\in(0,\pi)$,  let
$$
\Gamma_{\kappa, \theta}^\tau=\{z\in\mathbb C:~|z|=\kappa,~|{\rm arg} z|\leq  \theta\}\cup\Big\{z\in\mathbb C:~\kappa\leq | z|\leq \frac{\pi}{\tau\sin\theta},~|{\rm arg} z|=  \theta\Big\},
$$
it is clear that $\Gamma_{\kappa, \pi-\vartheta}^\tau\subset \Gamma_{\kappa, \pi-\vartheta}\in \Sigma_{\kappa,\vartheta}$. By a similar approach in \cite{G-M}, it follows that
\begin{align*}
u_N^n=&\frac{\tau}{2\pi i}\int_{\Gamma_{\kappa, \pi-\vartheta}^\tau} e^{zt_n} \widetilde{u_N}(e^{-z\tau})dz \\
=&\int_0^{t_n} \mathscr S _{A,N}(t_n-\upsilon)\overline{\mathscr F}_N(\upsilon)d\upsilon+\int_0^{t_n}\int_D \mathcal{F}_{W,P}(t_n-\upsilon,x,y)\gamma(\upsilon)\xi_{W}^{H_1,H_2}(dy,d\upsilon),
\end{align*}
where
$$
\mathscr S _{A,N}(t)=\frac{1}{2\pi i}\int_{\Gamma_{\kappa,\pi- \vartheta}^\tau}e^{z t}\mu^{-1}(\delta_\tau(e^{-z\tau}))^{\alpha-1}(h(\delta_\tau(e^{-z\tau}))+A_N^s))^{-1}\frac{z}{e^{z\tau}-1}dz.
$$
and
$$
\mathcal{F}_{W,P}(t ,x,y)= \sum_{k=1}^N \mathscr S _{A,N}(t )P_Ne_k(x)e_{ k}(y).
$$

In particular, for $u_N\in H_N$, we have
$$
\mathscr S _N(t)u_N=\sum_{k=1}^{N} u_{k,1}(t)(u_N,e_k)e_k,
$$
where
$$
u_{k,1}(t)=\frac{1}{2\pi i}\int_{\Gamma_{\kappa,\pi- \vartheta}^\tau}e^{z t}\mu^{-1}(\delta_\tau(e^{-z\tau}))^{\alpha-1}(h(\delta_\tau(e^{-z\tau}))+\rho_k^s))^{-1}\frac{z}{e^{z\tau}-1}dz.
$$
\begin{lemma}\cite{Gunzburger18}\label{Gun}
	Let $\alpha  \in  (0, 1)$ and $\theta  \in
	\bigl(  \pi /	2 , arccot	\bigl( -2/\pi	\bigr) \bigr) $,
	where arccot means the inverse function of cot, and a fixed $\xi  \in  (0, 1)$. Then, when
	$z$ lies in the region enclosed by $\Gamma^\tau_\xi
	= \{ z = -\ln (\xi )/\tau  + iy : y \in  \mathbb R,~ {\rm and}\ | y|  \leq  \pi /\tau \},$
	$\Gamma^\tau_{\theta,\kappa}$, and the two lines $\mathbb R \pm  i\pi /\tau$, whenever $0 < \kappa  \leq  \min(1/T,-\ln(\xi )/\tau )$, $\delta_\tau  (e^{-z\tau})$
	and $(\delta_\tau  (e^{-z\tau}) + A)^{-1}$ are both analytic. Additionally, there hold
	\begin{align*}
&\delta_\tau  (e^{-z\tau})\in \Sigma_\theta, \quad
c_0|z|\leq |\delta_\tau  (e^{-z\tau})|\leq c_1|z|,\\
&|\delta_\tau  (e^{-z\tau})-z|\leq c_2\tau|z|^2,\quad
|(\delta_\tau  (e^{-z\tau}))^\alpha-z^\alpha|\leq c_3\tau|z|^{\alpha+1},
	\end{align*}
	for all $z\in  \Gamma^\tau_{\theta,\kappa},$ where $\kappa  \in  (0, \min(1/T, -\ln(\xi )/\tau ))$ and the constants $c_i$, $i=0,\cdots,3$ are independent
	of $\tau$.
\end{lemma}
Let $U_{k,1}(t_0)=0$ and
$$
U_{k,1}(t)=\frac{1}{\tau}\int_{t_{j-1}}^{t_{j }} u_{k,1}(r)dr,\quad t\in [ t_{j-1 },t_{j}),
$$
and $U_{k,1}(t_w)=U_{k,1,w}$, $w=0,1,\cdots$. Thus, it follows that
\begin{align*}
	U_{k,1,w}=&\frac{1}{\tau}\frac{1}{2\pi i}\int_{t_w}^{t_{w+1}} \int_{\Gamma_{\kappa, \pi-\vartheta}^\tau} e^{zr}\widehat{u}_{k,1}(z)dzdr
	=  \frac{1}{2\pi i} \int_{\Gamma_{\kappa, \pi-\vartheta}^\tau} \frac{e^{zt_{w+1}-e^{zt_w}}}{z\tau} \widehat{u}_{k,1}(z)dz \\
	=&\frac{1}{2\pi i} \int_{\Gamma_{\kappa, \pi-\vartheta}^\tau}  e^{zt_w}  \mu^{-1}(\delta_\tau(e^{-z\tau}))^{\alpha-1}(h(\delta_\tau(e^{-z\tau}))+\rho_k^s))^{-1} dz.
\end{align*}
By the same manner in \cite[Proposition 3.2]{Gunzburger18}, we have
$$
\sum_{w=0}^\infty 	U_{k,1,w} \zeta^w=({\tau}\mu)^{-1}(\delta_\tau(\zeta))^{\alpha-1}(h(\delta_\tau(\zeta))+\rho_k^s))^{-1}.
$$
Moreover, the Laplace transform shows that
\begin{align*}
	\widetilde{U}_{k,1 }(z)=& \sum_{w=1}^\infty U_{k,1,w}\int_{t_w}^{t_{w+1}}e^{-zt} dt\\
=& \sum_{w=1}^\infty U_{k,1,w}e^{-zt_w}\frac{1-e^{-z\tau}}{z}= (\mu z)^{-1}(\delta_\tau(e^{-z\tau}))^{\alpha }(h(\delta_\tau(e^{-z\tau}))+\rho_k^s))^{-1} .
\end{align*}

\begin{lemma}\label{bouned}
Let
$$
h_\delta(z,\rho_k^s):= z^{\alpha-1}(h(z)+\rho_k^s)^{-1}-     (\delta_\tau(e^{-z\tau}))^{\alpha-1}(h(\delta_\tau(e^{-z\tau}))+\rho_k^s)^{-1}\frac{z\tau }{e^{z\tau}-1} ,
$$
then, we have
$$
\left| h_\delta(z,\rho_k^s) \right|\leq   \frac{c\mu |z|^{\alpha }\tau }{ |z|^\alpha+\rho_k^s} ,~~z\in \Gamma_{\kappa, \pi-\vartheta}^\tau.
$$
\end{lemma}
\begin{proof}
By the triangle inequality, we have
\begin{align*}
 	\left| h_\delta(z,\rho_k^s)\right|
\leq & \left| 1-\frac{z\tau }{e^{z\tau}-1}\right|	\left| z^{\alpha-1}(h(z)+\rho_k^s)^{-1}\right|\\
&+\left| \frac{z\tau }{e^{z\tau}-1}\right||z|^{\alpha-1} \left| (h(z)+\rho_k^s)^{-1}-(h(\delta_\tau(e^{-z\tau}))+\rho_k^s)^{-1}\right|\\
&+\left| \frac{z\tau }{e^{z\tau}-1}\right||z^{\alpha-1}-(\delta_\tau(e^{-z\tau}))^{\alpha-1}| \left|  (h(\delta_\tau(e^{-z\tau}))+\rho_k^s)^{-1}\right|.
\end{align*}
The Taylor expansion $\left| 1-\frac{z\tau }{e^{z\tau}-1}\right|\leq c|z\tau|$ and $\left| h(z)+\rho_k^s \right|\geq c\mu^{-1}(|z|^\alpha+\rho_k^s)$, it follows that
\begin{align*}
 \left| 1-\frac{z\tau }{e^{z\tau}-1}\right|	\left| z^{\alpha-1}(h(z)+\rho_k^s)^{-1}\right|\leq \frac{\mu |z|^{\alpha }\tau }{ |z|^\alpha+\rho_k^s}.
\end{align*}
In view of Lemma \ref{Gun}, form $\left|  \frac{z\tau }{e^{z\tau}-1}\right|\leq c $ for $z\in  \Gamma_{\kappa, \pi-\vartheta}^\tau$, the second term of above inequality is estimated as
\begin{align*}
& \left| \frac{z\tau }{e^{z\tau}-1}\right||z|^{\alpha-1} \left|h(z)-(h(\delta_\tau(e^{-z\tau}))\right| \left|(h(z)+\rho_k^s)^{-1} (h(\delta_\tau(e^{-z\tau}))+\rho_k^s)^{-1}\right|\\
\leq & c\tau |z|^{\alpha}(|h(z)|+|h(\delta_\tau(e^{-z\tau}))|)\left|(h(z)+\rho_k^s)^{-1} (h(\delta_\tau(e^{-z\tau}))+\rho_k^s)^{-1}\right|\\
\leq& c\tau |z|^{\alpha}
(|(h(\delta_\tau(e^{-z\tau}))+\rho_k^s)^{-1}|+ |(h(z)+\rho_k^s)^{-1} |)\\
\leq& \frac{c\mu |z|^{\alpha }\tau }{ |z|^\alpha+\rho_k^s}.
\end{align*}
Since the angle condition
arg($z^\alpha) \leq \alpha(\pi-\vartheta)  <\pi$ and arg($\delta_\tau(e^{-z\tau})^\alpha  )\leq \alpha(\pi-\vartheta)  <\pi$ satisfy the requirements in Lemme \ref{Gun}, it follows that
\begin{align*}
 &\left| \frac{z\tau }{e^{z\tau}-1}\right||z^{\alpha-1}-(\delta_\tau(e^{-z\tau}))^{\alpha-1}| \left|  (h(\delta_\tau(e^{-z\tau}))+\rho_k^s)^{-1}\right|\\
 \leq& c(|z^{\alpha }-(\delta_\tau(e^{-z\tau}))^{\alpha }| |z|^{-1}+|z^{-1}-\delta_\tau(e^{-z\tau}))^{-1}| |\delta_\tau(e^{-z\tau})) |^\alpha) \left|  (h(\delta_\tau(e^{-z\tau}))+\rho_k^s)^{-1}\right|\\
 \leq &c \tau|z|^{\alpha} \left|  (h(\delta_\tau(e^{-z\tau}))+\rho_k^s)^{-1}\right|\\
 \leq& \frac{c\mu |z|^{\alpha }\tau }{ |z|^\alpha+\rho_k^s}.
\end{align*}
Together with above arguments, the proof is complete.

\end{proof}

\begin{lemma}\label{12121}
 For any $\tau  < \tau^\ast$  (the value of $\tau^\ast$  depends on $\rho_k^s$), there holds
$$
|\rho_k^{s\eta}  (\mu z)^{-1}(\delta_\tau(e^{-z\tau}))^{\alpha }(h(\delta_\tau(e^{-z\tau}))+\rho_k^s))^{-1}|\leq c|z|^{\alpha\eta-1}  e^{\eta \alpha|z|\tau},~~z\in \Gamma_{\kappa, \pi-\vartheta}\setminus \Gamma_{\kappa, \pi-\vartheta}^\tau,
$$
for $\eta\in[0,1]$, $s\in(0,1)$.
\end{lemma}
\begin{proof}
Let $z\in \Gamma_{\kappa, \pi-\vartheta}\setminus \Gamma_{\kappa, \pi-\vartheta}^\tau$, the Lemma 4.2 in \cite{Nie}, we have
$$
|\rho_k^s  z ^{-1}(\delta_\tau(e^{-z\tau}))^{\alpha }( (\delta_\tau(e^{-z\tau}))^\alpha+\rho_k^s))^{-1}|\leq c|z|^{\alpha\eta-1}  e^{\alpha\eta|z|\tau}.
$$
Moreover, since $|\delta_\tau(e^{-z\tau})|\geq\frac{e-1}{\tau}$  and $|\delta_\tau(e^{-z\tau})|\leq |z|e^{|z|\tau}$ by \cite{Nie}, let $\tau$  be small enough to satisfy $(\frac{e-1}{\tau})^\alpha>2\rho_k^s$, from $|h(\delta_\tau(e^{-z\tau}))|\geq 
 |\delta_\tau(e^{-z\tau})|^\alpha$ it follows that
$$
\big|\rho_k^{s\eta} \delta_\tau(e^{-z\tau})(h(\delta_\tau(e^{-z\tau}))+\rho_k^s)^{-1} \big|\leq  c |z|^{\alpha\eta} e^{\alpha\eta|z|\tau}.
$$
We thus obtain the desired result.
\end{proof}

\begin{theorem}\label{thm82}
	Let $u_N(t_n)$ and $u_N^n$ be the solutions of (\ref{approximation}) and (\ref{approximation2}), respectively.  Let $s>\frac{1-H_1}{2}$, $2sH_2/\alpha+ H_1-1 >0$, then for small $\epsilon\in (0,2H_2+ {(H_1-1 )\alpha}/{s})$ there
	holds
$$ \mathbb E\| u_N(t_n)- u^n_N\|^2
	\leq c \tau^{ 2H_2+{\alpha( H_1-1 )}/{s}-\epsilon}.
	$$
\end{theorem}

\begin{proof}
	Let $e_N^n =u_N(t_n)- u^n_N$, from (\ref{120}) and (\ref{121}), we get
\begin{align*}
\mathbb E\| e_N ^n\|^2
  \leq& 3\mathbb E\left\| \int_0^{t_n}S_N(t_n-\upsilon)\mathscr{F}_N(\upsilon)-\mathscr{S}_N(t_n-\upsilon)\overline{\mathscr F}_N(\upsilon)d\upsilon\right\|^2\\
  &+3\mathbb E\bigg\| \int_0^{t_n} \int_D  (F_{W,P}(t_n-\upsilon,x,y)-\mathcal{F}_{W,P}(t_n-\upsilon,x,y))  \gamma(\upsilon)\xi^{H_1,H_{2}}(dy, d\upsilon)\bigg\|^2\\
  &+3\mathbb E\bigg\| \int_0^{t_n}\int_D \mathcal{F}_{W, P}(t_n-\upsilon,x,y) \gamma(\upsilon) (\xi^{H_1,H_{2}}(dy, d\upsilon)-\xi_{W}^{H_1,H  _2}(dy,d\upsilon)) \bigg\|^2\\
  \leq &3 F_1+3F_2+3F_3.
\end{align*}	

To estimate $F_1$, we have
\begin{align*}
F_1\leq &3\mathbb E\left\|\sum_{i=1}^n \int_{t_{i-1}}^{t_i}S_N(t_n-\upsilon)(\mathscr{F}_N(\upsilon)-  {\mathscr F}_N(t_{i-1}))d\upsilon\right\|^2\\
&+3\mathbb E\left\|\sum_{i=1}^n \int_{t_{i-1}}^{t_i}(S_N(t_n-\upsilon) -\mathscr{S}_N(t_n-\upsilon)) {\mathscr F}_N(t_{i-1})d\upsilon\right\|^2\\
&+3\mathbb E\left\|\sum_{i=1}^n \int_{t_{i-1}}^{t_i} \mathscr{S}_N(t_n-\upsilon)({\mathscr{F}}_N(t_{i-1})-\overline{\mathscr F}_N(t_i))d\upsilon\right\|^2\\
\leq &3 F_{11}+ 3F_{12}+ 3F_{13}.
\end{align*}	
Theorem \ref{thm3} shows that
\begin{align*}
	F_{11}\leq & c\sum_{i=1}^n \int_{t_{i-1}}^{t_i}\mathbb E \|u_N(\upsilon)-  u_N(t_{i-1})\|^2d\upsilon
	\leq  c \tau^{2H_2+(H_1-1)\alpha/s}.
\end{align*}
From Lemma \ref{bouned}, we note that for any $v\in H_N$,
\begin{align*}
&	\| (S_N( \upsilon) -\mathscr{S}_N( \upsilon))P_Nv \|^2 \\ \leq &c\sum_{k=1}^N \bigg|\int_{\Gamma_{\kappa, \pi-\vartheta}\setminus \Gamma_{\kappa, \pi-\vartheta}^\tau} e^{z\upsilon}\mu^{-1}z^{\alpha-1}(h(z)+\rho_k^s)^{-1}dz\bigg |^2v_k^2
 +	c\sum_{k=1}^N \bigg| \int_{\Gamma_{\kappa, \pi-\vartheta}^\tau} e^{z\upsilon} \mu^{-1} h_\delta(z,\rho_k^s) \bigg |^2v_k^2\\
\leq& c\sum_{k=1}^N \bigg(\int_{\Gamma_{\kappa, \pi-\vartheta}\setminus \Gamma_{\kappa, \pi-\vartheta}^\tau}  |e^{z\upsilon} z^{\alpha-1}|(|z|^\alpha+\rho_k^s)^{-1} | dz|\bigg )^2v_k^2\\
&+	c\tau^2 \sum_{k=1}^N \bigg( \int_{\Gamma_{\kappa, \pi-\vartheta}^\tau} |e^{z\upsilon}| |z|^{\alpha }(|z|^\alpha+\rho_k^s)^{-1} ||dz|\bigg )^2v_k^2\\
\leq& c\sum_{k=1}^N  \bigg(\int_{\Gamma_{\kappa, \pi-\vartheta}\setminus \Gamma_{\kappa, \pi-\vartheta}^\tau}  |e^{z\upsilon}|    |z|^{  -1 }  | dz| \bigg)^2 v_k^2
 +	c\tau^2 \sum_{k=1}^N  \bigg( \int_{\Gamma_{\kappa, \pi-\vartheta}^\tau} |e^{ z\upsilon}|  |dz|  \bigg)^2 v_k^2\\
\leq & c\left( \tau^{2-2\epsilon} \bigg(\int_{\Gamma_{\kappa, \pi-\vartheta}\setminus \Gamma_{\kappa, \pi-\vartheta}^\tau}  |e^{z\upsilon}|    |z|^{  -\epsilon }  | dz| \bigg)^2+	c\tau^{2-2\epsilon }   \int_{\Gamma_{\kappa, \pi-\vartheta}^\tau} |e^{ z\upsilon}|^2 |z|^{1-2\epsilon }  |dz|   \right) \| P_Nv\|,
\end{align*}	
where we used the inequalities $
|z|^\alpha(|z|^\alpha+\rho_k^s)^{-1}\leq  1$,
and
$$
\int_{\Gamma_{\kappa, \pi-\vartheta}^\tau}   |z|^{2\epsilon -1} ||dz|\leq c\int_\kappa^{\frac{\pi}{\tau\sin(\vartheta)}} r^{2\epsilon -1 } dr+\int_{-(\pi-\vartheta)}^{\pi-\vartheta} \kappa^{ 2\epsilon }d\psi\leq c\tau^{-2\epsilon }.
$$
Therefore, by Cauchy-Schwarz inequality, for $\epsilon<1/2$,
the estimate of $F_{12}$ is given by
\begin{align*}
F_{12}\leq &\mathbb E\sum_{i=1}^{n}\int_{t_{i-1}}^{t_i} (t_n-\upsilon)^{1-\epsilon} \|  S_N(t_n-\upsilon) -\mathscr{S}_N(t_n-\upsilon)\|^2\| {\mathscr F}_N(t_{i-1})\|^2d\upsilon \\
\leq &c \int_{0}^{t_n} (t_n-\upsilon)^{1-\epsilon} \|  S_N(t_n-\upsilon) -\mathscr{S}_N(t_n-\upsilon)\|^2\mathbb E\| u(\upsilon)\|^2d\upsilon \\
\leq& c\tau^{2-2\epsilon} \int_0^{t_n}  \upsilon ^{ 1-\epsilon }  \bigg( \int_{\Gamma_{\kappa, \pi-\vartheta}\setminus \Gamma_{\kappa, \pi-\vartheta}^\tau}    |e^{z\upsilon}|    |z|^{  -\epsilon }  | dz| \bigg)^{ 2} d\upsilon    \\
&+ 	c\tau^{2-2\epsilon} \int_0^{t_n}  \upsilon ^{1-\epsilon} \int_{\Gamma_{\kappa, \pi-\vartheta}^\tau} |e^{ z \upsilon  }|^2 |z|^{1-2\epsilon } ||dz|  \\
\leq& c\tau^{2-2\epsilon}.
\end{align*}	
The estimate of $F_{13}$ is given by
$$
F_{13}\leq c\tau \sum_{i=1}^{n-1}\mathbb E\|u_N(t_{i})-u_N^i\|^2=c\tau \sum_{i=1}^{n-1}\mathbb E\|e_N ^i\|^2 .
$$

To estimate $F_2$, from Remark \ref{Re-2}, we first note that
\begin{align*}
\int_0^{t_n} \bigg | \partial_\upsilon^{\frac{1-2H_2}{2}} \rho_k^{sw}( u_k(t_n-\upsilon)& -u_{k,1}(t_n-\upsilon))  \bigg |^2d\upsilon
  \leq  c\|    u_{k,2}(t_n-\upsilon)   \|_{H^{\frac{1-2H_2}{2}} (0,t_n)}^2 \\
 &+c\|   u_{k,2}(t_n-\upsilon)-u_{k,1}(t_n-\upsilon)  \|_{H^{\frac{1-2H_2}{2}} (0,t_n)}^2\\
  \leq & c\left \|   \int_{\Gamma_{\kappa,\pi- \vartheta}\setminus\Gamma_{\kappa,\pi- \vartheta}^\tau}|e^{z \upsilon}|  |z|^{\frac{1-2H_2}{2} }\widetilde{u}_{k,2}(z) |dz| \right\| ^2 _{L^2(0,t_n)}\\
&+c\left \|  \int_{ \Gamma_{\kappa,\pi- \vartheta}^\tau}|e^{z \upsilon}|  |z|^{\frac{1-2H_2}{2} } |\widetilde{u}_{k,2}(z)-\widetilde{u}_{k,1}(z) | |dz|\right\| ^2 _{L^2(0,t_n)}\\ \leq & F_{21}+F_{22},
\end{align*}
where
\begin{align*}
u_{k,2}(t)=&\frac{1}{2\pi i}\int_{\Gamma_{\kappa,\pi- \vartheta}\setminus\Gamma_{\kappa,\pi- \vartheta}^\tau}e^{z t} \rho_k^{sw}\mu^{-1} z^{\alpha-1}(h(z)+\rho_k^s))^{-1} dz,\\
u_{k,2}(t )-u_{k,1}(t ) =& \frac{1}{2\pi i}\int_{ \Gamma_{\kappa,\pi- \vartheta}^\tau}e^{z t}\rho_k^{sw}\mu^{-1}  h_\delta(z,\rho_k^s) dz.
\end{align*}
Therefore, by $\rho_k^{sw}|z|^\alpha(|z|^\alpha+\rho_k^s)^{-1}\leq  |z|^{\alpha w}$ for $w\in[0,1]$, since $u_{k,2}(0)$ is finite, we have
\begin{align*}
F_{21}\leq &  c\int_0^{t_n} \left( \int_{\Gamma_{\kappa,\pi- \vartheta}\setminus\Gamma_{\kappa,\pi- \vartheta}^\tau}|e^{z \upsilon}|  |z|^{\frac{1-2H_2}{2}+\alpha w-1} |dz|\right)^2d\upsilon\\
\leq & c\int_0^{t_n} \left(\int_{1/\tau}^\infty e^{ r\cos(\pi-\vartheta)\upsilon} r^{\frac{1-2H_2}{2}+\alpha w-1}dr\right)^2d\upsilon\\
\leq & c\int_0^{t_n} \left(\tau^{H_2-\alpha w-\frac{\epsilon}{2}}  \upsilon^{\frac{\epsilon-1}{2}} \int_{0}^\infty e^{- 2r\cos(\vartheta) } r^{   -\frac{\epsilon+1}{2}}dr\right)^2d\upsilon\\
\leq & c\tau^{2(H_2-\alpha w)- \epsilon  } ,
\end{align*}
for requiring $\alpha w<H_2$ and $\epsilon/2\in(0,H_2-\alpha w).$ Similarly, by using
\begin{equation}\label{Gamma}
\begin{aligned}
	\int_{\Gamma_{\kappa, \pi-\vartheta}^\tau}   |e^{z\upsilon}|^{2}|z|^{ \zeta} ||dz|
	\leq &  c \int_{\kappa}^{\frac{\pi}{\tau\sin(\vartheta)}} e^{-2r \cos(\vartheta)\upsilon}r^{\zeta}dr+c\int_{-(\pi-\vartheta)}^{\pi-\vartheta} e^{2\kappa \cos(\psi)\upsilon} \kappa^{\zeta+1}d\psi\\
	\leq &  c\tau^{- \zeta-\epsilon }\left( \int_{\kappa\upsilon}^{\frac{\pi\upsilon}{\tau\sin(\vartheta)}} e^{-2r \cos(\vartheta)\upsilon }r^{-\epsilon}dr+c  e^{2\kappa  \upsilon} \kappa^{ 1-\epsilon}\right) \\
	\leq &  c\tau^{- \zeta-\epsilon } \upsilon^{ \epsilon-1} ,
\end{aligned}
\end{equation}
for $\epsilon\in(0,1-2(H_2-\alpha w))$, $\kappa t_n<\pi/\sin(\vartheta)$ and $\upsilon\in(0,t_n]$ implying $\upsilon\kappa< \pi/\sin(\vartheta)$, we have
\begin{align*}
F_{22}\leq &  c\int_0^{t_n}  \tau^2\left(\int_{ \Gamma_{\kappa,\pi- \vartheta}^\tau}|e^{z \upsilon}|  |z|^{\frac{1-2H_2}{2}+\alpha w } |dz|\right)^2d\upsilon\\
	\leq & c\tau^2 \int_0^{t_n} \int_{ \Gamma_{\kappa,\pi- \vartheta}^\tau}|e^{z \upsilon}| ^2 |z|^{ {1-2H_2}+2\alpha w } |dz| \int_{ \Gamma_{\kappa,\pi- \vartheta}^\tau}  |dz|d\upsilon\\
\leq &  c\tau \int_0^{t_n}   \int_{ \Gamma_{\kappa,\pi- \vartheta}^\tau}|e^{z \upsilon}| ^2 |z|^{ {1-2H_2}+2\alpha w } |dz|d\upsilon \leq c\tau^{2(H_2-\alpha w)-\epsilon } .
\end{align*}
This leads to
\begin{align*}
	F_2\leq & c\sum_{k=1}^N \int_0^{t_n}  \bigg | \partial_\upsilon^{\frac{1-2H_2}{2}} (u_{k}(t_n-\upsilon)-u_{k,1}(t_n-\upsilon))  \bigg |^2d\upsilon\|e_k\|_{H_0^{\frac{1-2H_1}{2}}(D)} \\
	\leq& c\sum_{k=1}^N \rho_k^{-\frac{1}{2}-\epsilon'}\int_0^{t_n}  \bigg | \partial_\upsilon^{\frac{1-2H_2}{2}} \rho_k^{\frac{1-H_1}{2}+\frac{\epsilon'}{2} } (u_{k}(t_n-\upsilon)-u_{k,1}(t_n-\upsilon))  \bigg |^2d\upsilon \\
	\leq&c \sup_{1\leq k\leq N} \int_0^{t_n}  \bigg | \partial_\upsilon^{\frac{1-2H_2}{2}} \rho_k^{\frac{1-H_1}{2}+\frac{\epsilon'}{2} } (u_{k}(t_n-\upsilon)-u_{k,1}(t_n-\upsilon))  \bigg |^2d\upsilon\\
	\leq& c \tau^{2 H_2- {\alpha(1-H_1)}/{ s}  -\alpha\epsilon'/s- \epsilon  }	\leq  c \tau^{2 H_2- {\alpha(1-H_1)}/{ s}   - 2\epsilon  },
\end{align*}
for some $\epsilon'\in (s\epsilon/\alpha,1)$.

Note that by Remark \ref{Re-2},
\begin{align*}
\int_0^{t_n}\big |\partial_\upsilon^{\frac{1-2H_2}{2}} l_{t_n}(\upsilon) \big |^2d\upsilon
	=  \big \|  l_{t_n}(\cdot)\big \|^2_{H_0^\frac{1-2H_2}{2}(0,t_n)} ,
\end{align*}	
where
$$
l_{t_n}(\upsilon):= \left(u_{k,1} (t_n-\upsilon )  -\frac{1}{\tau}\sum_{j=1}^n\chi_{(t_{j-1},t_j]}(\upsilon)\int_{t_{j-1}}^{t_j}u_{k,1} (t_n-\xi ) d\xi\right)\gamma(\upsilon) ,~~\upsilon\in [0,t_n).
$$
To estimate $F_3$, we have
\begin{align*}
F_3 
	\leq &   \mathbb E\bigg\| \int_0^{t_n}\sum_{k=1}^N \bigg(\mathscr{S}_N (t_n-\upsilon )-\frac{1}{\tau}\sum_{j=1}^n\chi_{(t_{j-1},t_j]}(\upsilon)\int_{t_{j-1}}^{t_j}\mathscr{S}_N (t_n-\xi ) d\xi \bigg)\\
	&\cdot   P_Ne_k(x)e_{ k}(y)\xi^{H_1,H_2}(dy,d\upsilon) \bigg\|^2\\
	\leq& c\sum_{k=1}^N\rho_k^{-\frac{1}{2}-\epsilon''} \int_0^{t_n}\bigg|\partial_\upsilon^{\frac{1-2H_2}{2}} \rho_k^{\frac{1-H_1+\epsilon''}{2}}l_{t_n}(\upsilon)\bigg|^2d\upsilon  \\
\leq& c \sup_{1\leq k\leq N} \int_0^{t_n}\bigg|\partial_\upsilon^{\frac{1-2H_2}{2}} \rho_k^{\frac{1-H_1+\epsilon''}{2}} \big (u_{k,1} ( \upsilon )  -U_{k,1} (\upsilon  )  \big ) \gamma(\upsilon)\bigg|^2d\upsilon\\
\leq& c \sup_{1\leq k\leq N}\bigg( \int_0^{t_1}\bigg|\partial_\upsilon^{\frac{1-2H_2}{2}} \rho_k^{\frac{1-H_1+\epsilon''}{2}} \big (u_{k,1} ( \upsilon )  -U_{k,1} (\upsilon  )  \big ) \gamma(\upsilon)\bigg|^2d\upsilon\\
&+\int_ {t_1}^{t_n}\bigg|  \rho_k^{\frac{1-H_1+\epsilon''}{2}}  \int_{\Gamma_{\kappa,\pi- \vartheta}\setminus\Gamma_{\kappa,\pi- \vartheta}^\tau}   e^{z\upsilon} z^{\frac{1-2H_2}{2} } \widehat{U}_{k,1} (z)  dz\bigg|^2d\upsilon  \\
&+\int_ {t_1}^{t_n}\bigg|  \rho_k^{\frac{1-H_1+\epsilon''}{2}}  \int_{ \Gamma_{\kappa,\pi- \vartheta}^\tau}   e^{z\upsilon} z^{\frac{1-2H_2}{2} }\big (\widehat{u}_{k,1} ( z)  -\widehat{U}_{k,1} (z)  \big ) dz\bigg|^2d\upsilon \bigg)
 \leq F_{31}+F_{32}+F_{33},
\end{align*}	
for some small $\epsilon''>0$. To estimate $F_{31}$, we note that $t_1=\tau$ and then for $H_2=1/2$, the mean value theorem shows
\begin{align*}
F_{31}\leq & \frac{c}{\tau^2}\int_0^{t_1} \bigg|\int_0^{t_1} \rho_k^{\frac{1-H_1+\epsilon''}{2}}   u_{k,1}' ( \upsilon_\xi )\chi_{(\xi,t_1)}(\upsilon_\xi)(t_1-\xi)d\xi    \bigg|^2d\upsilon\\
\leq & c\tau^{1+2\epsilon''-\alpha (1-H_1)/s -\alpha\epsilon''/s },
\end{align*}
since
$$
|\rho_k^{\frac{1-H_1+\epsilon''}{2}}   u_{k,1}' ( t)|\leq c\tau^{-\alpha (1-H_1)/2s-\alpha\epsilon''/2s}t^{\epsilon''-1},
$$
from $|\widetilde{u}_{k,1}'(z)|\leq c |e^{zt}| |z|^{\alpha}(|z|^\alpha+\rho_k^s)^{-1},
$ $z\in \Gamma_{\kappa, \pi-\vartheta}^\tau$. Similarly, for $H_2\in(0,1/2)$, we have
\begin{align*}
	F_{31}\leq & c\int_0^{t_1} \bigg|\int_0^{r} k_{H_2-\frac{1}{2}}(r-\upsilon) \rho_k^{\frac{1-H_1+\epsilon''}{2}}   \frac{d}{d\upsilon}(u_{k,1} ( \upsilon  ) -u_{k,1} (0 ) ) d\upsilon  \bigg|^2dr\\
	&+c\int_0^{t_1} \bigg|\frac{d}{dr}\int_0^{r} k_{H_2-\frac{1}{2}}(r-\upsilon) \rho_k^{\frac{1-H_1+\epsilon''}{2}}    (u_{k,1} (0 )-U_{k,1} ( \upsilon  )  ) d\upsilon  \bigg|^2dr\\
	\leq & c\tau^{2H_2+2\epsilon''-\alpha (1-H_1)/s -\alpha\epsilon''/s }.
\end{align*}

Note that for $z\in  \Gamma_{\kappa,\pi- \vartheta}\setminus\Gamma_{\kappa,\pi- \vartheta}^\tau$, Lemma \ref{12121} implies
\begin{align*}
\rho_k^{\frac{1-H_1+\epsilon''}{2}}| (\delta_\tau(e^{-z\tau}))^{\alpha }(h(\delta_\tau(e^{-z\tau}))+\lambda_k^s) ^{-1}|\leq & \rho_k^{\frac{1-H_1+\epsilon''}{2}}  |\delta_\tau(e^{-z\tau})|^{\alpha }( |\delta_\tau(e^{-z\tau})|^\alpha+\lambda_k^s) ^{-1}\\
\leq & |z|^{\alpha(1-H_1+\epsilon'')/2s}e^{\alpha(1-H_1+\epsilon'')|z|\tau/2s}.
\end{align*}
To estimate $F_{32}$, we have
\begin{align*}
	F_{32}\leq &\int_ {t_1}^{t_n}\bigg|    \int_{\Gamma_{\kappa,\pi- \vartheta}\setminus\Gamma_{\kappa,\pi- \vartheta}^\tau}   e^{|z|\upsilon+\alpha(1-H_1+\epsilon'')|z|\tau/2s} z^{\frac{1-2H_2}{2}+\alpha(1-H_1+\epsilon'')/2s-1 }  dz\bigg|^2d\upsilon\\
	\leq &c \tau^{2H_2-\alpha(1-H_1+\epsilon'')/ s-\epsilon} \int_ {t_1}^{t_n}   \int_{\Gamma_{\kappa,\pi- \vartheta}\setminus\Gamma_{\kappa,\pi- \vartheta}^\tau}   e^{2|z|\upsilon } |z|^{-\epsilon }  |dz| d\upsilon\\
	\leq&c\tau^{2H_2-\alpha(1-H_1+\epsilon'')/ s-\epsilon}.
\end{align*}	
To estimate $F_{33}$, Lemma \ref{Gun} and (\ref{Gamma}) show that
\begin{align*}
	F_{33}\leq &
\int_ {t_1}^{t_n}\bigg(  \rho_k^{\frac{1-H_1+\epsilon''}{2}}  \int_{ \Gamma_{\kappa,\pi- \vartheta}^\tau}  | e^{z\upsilon}| |z|^{\frac{1-2H_2}{2} } |\delta_\tau(e^{-z\tau})|^{\alpha-1} |(h(\delta_\tau(e^{-z\tau}))+\rho_k^s)^{-1}| |dz|\bigg)^2d\upsilon \\
\leq & \tau^{1+2H_2-\frac{\alpha(1-H_1+\epsilon'')}{s}-\epsilon}.
\end{align*}	
Together above arguments, the discrete Gronwall inequality implies the desired results.
\end{proof}

Combining the conclusions of Theorem  \ref{thm8} and Theorem \ref{thm82}, we have the following result.
\begin{theorem}\label{thm68}
	Let $u$ and  $u_N^n$ be the solutions of problem (\ref{P}) and (\ref{approximation2}), respectively.  Let $s>\frac{1-H_1}{2}$, $2sH_2/\alpha+ H_1-1 >0$, $2\sigma<\min\{2sH_2/\alpha+H_1-1,2s+H_1-1\}$, then for small $\epsilon\in (0,2H_2+ {(H_1-1 )\alpha}/{s})$ there
	holds
$$ \mathbb E\| u (t_n)- u^n_N\|^2
	\leq c ((N+1)^{-4\sigma}+\tau^{ 2H_2+{\alpha( H_1-1 )}/{s}-\epsilon}).
	$$
\end{theorem}

Theorem \ref{thm5} provides the regularization error, which measures the cost of approximating the rough noise by a smooth Wong-Zakai process. This estimate is then used in the proof of Theorem \ref{thm5-8}, together with the discretization error from Theorem \ref{thm8}, to obtain the final convergence result.  Furthermore,
Combining Theorem  \ref{thm5}, and Theorems \ref{thm5-8}-\ref{thm82}, another version of Theorem \ref{thm68} is extra added a term $N^{1-2H_1}\tau^{2H_2-\frac{\alpha}{2s}-\epsilon }$ in estimate $\mathbb{E}\|u(t_n)-u_N^n\|^2$.

\section{Numerical experiments}\label{sect6}

This section will present numerical experiments to verify the theoretical results. In the numerical experiments, we next choose $f(u)=\sin(u)$ in equation (\ref{P}) with
$D=(0,1)$, and then $\rho_k=k^2\pi^2$, $k=1,2,\cdots$, $e_k(x)=\sqrt{2}/2\sin(k\pi x)$. Since the exact solution is hard to calculate, we will take $\widetilde{m}$ trajectories $\{\omega_i\}_{i=1}^{\widetilde{m}}$ to calculate this solution. Now, let us measure  the spatial and
temporal convergence rates by
$$
{\rm Rate}=\frac{\ln e_N/e_{2N}}{\ln 2},\quad {\rm Rate}=\frac{\ln e_\tau/e_{\tau/2}}{\ln 2},
$$
where
$$
e_N=\left(\frac{1}{\widetilde{m}}\sum_{i=1}^{\widetilde{m}} \|u_N^M(w_i)-u_{2N}^M(w_i)\|^2_{L^2(D)} \right)^{1/2},$$
$$e_\tau=\left(\frac{1}{\widetilde{m}}\sum_{i=1}^{\widetilde{m}}\|u_\tau (w_i)-u_{\tau/2} (w_i)\|^2_{L^2(D)} \right)^{1/2},
$$
here $u_N^M(\omega_i)$ and $u_\tau (\omega_i)$ denote the numerical solutions of $u$ at time $t_M$ with $N$ spectral bases and step size $\tau$ respectively, both evaluated at the $i$-th sample point \( \omega_i \).

In the numerical experiments, we take $T=0.1$, $\gamma(t)=1-\exp(-t)$, $\lambda=1$, and $\mu=0.8$.
\begin{example}
 %
In this example to show the temporal convergence, to show the temporal convergence, we fix the number of spectral modes at $N=1024$. 
For each time step size, we generate $\widetilde{m}=500$ independent noise sample paths. The numerical results with different $\alpha,\beta,s$, $H_1$ and $H_2$ are given in Table 1, all satisfying the theoretical assumptions $s > (1-H_1)/2$ and $2sH_2/\alpha + H_1 - 1 > 0$. As one can see, the numerical convergence rates can achieve $O(\tau^{2H_2+\alpha(H_1-1)/s-\epsilon})$, which agree with the predicted ones stated in Theorem \ref{thm82}.
\end{example}

\begin{table}[htbp]
	\centering
	\caption{Temporal errors and convergence rates.}
	\setlength{\tabcolsep}{2.9pt} 
	\renewcommand{\arraystretch}{0.5} 
	\begin{tabular}{@{}>{\centering\arraybackslash}m{4cm} *{5}{>{\centering\arraybackslash}m{1.6cm}}>{\centering\arraybackslash}m{2.9cm}@{}}
		\toprule
		$(\alpha,\beta,s,H_1,H_2) \backslash T / \tau$ &8&16&32&64&128      & Rate \\
		\midrule
(0.05,0.05,0.9,0.45,0.45) &4.380e-04 & 2.319e-04& 1.266e-04&	7.256e-05& 4.040e-05   & $\approx$ 0.8670(0.8694)
		\\
\addlinespace
(0.1,0.05,0.6,0.35,0.35)   & 1.974e-03	&1.261e-03	&8.333e-04&	5.703e-04	&3.941e-04	 	 & $\approx$ 0.6175(0.5917) \\
\addlinespace
(0.15,0.1,0.5,0.3, 0.4)   & 3.693e-03 & 2.482e-03 & 1.752e-03 & 1.267e-03 & 9.006e-04   & $\approx$ 0.5686(0.5900)\\
\addlinespace
(0.15,0.1,0.8,0.45,0.45)   & 1.226e-03  & 7.906e-04 & 5.434e-04 & 3.821e-04 & 2.629e-04   & $\approx$ 0.6058(0.6969)\\
\addlinespace
(0.2,0.1,0.75,0.45,0.4)   & 4.332e-07  &2.944e-07 &2.105e-07 &1.524e-07 &1.075e-07  & $\approx$ 0.5676(0.6533)\\
\addlinespace
(0.3,0.2,0.65,0.25,0.4)   & 5.482e-03 &	4.243e-03&	3.329e-03&	2.590e-03&	1.947e-03    & $\approx$ 0.4597(0.4538)
		\\
 \addlinespace
(0.4,0.2,0.75,0.3,0.4)   & 6.019e-03 &	4.825e-03&	3.879e-03&	3.044e-03&	2.307e-03  & $\approx$ 0.4360(0.4267)
		\\
\addlinespace
(0.5,0.3,0.6,0.45,0.4)   & 	9.233e-03 &	7.370e-03&	5.873e-03	&4.594e-03&	3.468e-03  & $\approx$ 0.4467(0.4417)
		\\
\addlinespace
(0.55,0.1,0.8,0.5,0.4)   & 6.108e-03 & 4.906e-03 & 3.900e-03 & 3.053e-03 & 2.311e-03   & $\approx$ 0.4457(0.4563)
		\\
\addlinespace
(0.6,0.1,0.8,0.5,0.4)   & 	7.162e-03 &	5.811e-03&	4.691e-03&	3.705e-03&	2.825e-03&     $\approx$ 0.4289(0.4250)
		\\
\addlinespace
(0.6,0.4,0.7,0.4,0.5)   & 8.113e-03 &	6.437e-03&	5.088e-03&	3.943e-03&	2.949e-03&  $\approx$ 0.4618(0.4857)
		\\
\addlinespace
(0.7,0.1,0.8,0.4,0.4)   & 	1.320e-02 &	1.126e-02&	9.491e-03&	7.815e-03&	6.173e-03   & $\approx$ 0.3720(0.2750)\\
		\bottomrule
	\end{tabular}
\label{tab:temporal}
\end{table}

\begin{example}

In this example, to  show the spatial convergence, we fix the time step size at $\tau = T/512$ (i.e., $M=512$ time steps). 
For each $N$, we generate $\widetilde{m}=500$ independent noise sample paths. The numerical results with different $\alpha,\beta,s$, $H_1$ and $H_2$ are given in Table 2. As one can see, the numerical convergence rates can achieve $O((N+1)^{-4\sigma })$ for $0<\sigma<0.5*\min\{2sH_2/\alpha + H_1 - 1, 2s + H_1 - 1\}$,  the numerical convergence rates are agree with the predicted ones stated in Theorem \ref{thm8}.
\end{example}

 \begin{table}[htbp]
	\centering
	\caption{  Spatial errors and convergence rates.}
	\setlength{\tabcolsep}{2.8pt} 
	\renewcommand{\arraystretch}{0.5} 
	\begin{tabular}{@{}>{\centering\arraybackslash}m{3.8cm} *{5}{>{\centering\arraybackslash}m{1.6cm}}>{\centering\arraybackslash}m{3cm}@{}}
		\toprule
		$(\alpha,\beta,s, H_1, H_2) \backslash N$ & 16 & 32 & 64 & 128 &256 &Rate \\
		\midrule
$(0.1,0.05,0.8,0.45,0.45)$  &3.490e-03 &  1.728e-03 & 8.433e-04 & 4.072e-04 & 1.969e-04    & $\approx$ 1.1055(1.0500)\\
\addlinespace
$(0.2,0.1,0.9,0.5,0.5)$  & 1.419e-03 &	5.936e-04 &	2.442e-04	& 9.911e-05&	4.039e-05 & $\approx$  1.3399(1.3000)\\
\addlinespace
$(0.25,0.1,0.6,0.35,0.35)$   &  6.914e-02	&4.823e-02&	3.311e-02&	2.245e-02&	1.505e-02
& $\approx$ 0.6683(0.5500)\\
\addlinespace
$(0.3,0.2,0.7,0.3,0.4)$    & 3.328e-02 &	2.095e-02 &1.301e-02&8.016e-03&4.921e-03 & $\approx$ 0.7767(0.7000)\\
\addlinespace
$(0.3,0.2,0.6,0.23,0.5)$   &  7.913e-02	&5.989e-02	&4.458e-02&	3.301e-02&	2.403e-02 & $\approx$ 0.5499(0.4300) \\
\addlinespace
$(0.4, 0.1, 0.75, 0.3, 0.4)$    &  3.360e-02 &	2.002e-02&	1.166e-02&	6.726e-03&	3.871e-03 &	 $\approx$	0.8596(0.8000)\\
\addlinespace
$(0.4,0.2,0.7,0.5,0.42)$   &  2.261e-02&	1.262e-02&	6.863e-03&	3.683e-03&	1.964e-03	& $\approx$ 0.9761(0.9000)\\
\addlinespace
$(0.4, 0.2, 0.88, 0.5, 0.4)$    &  5.177e-03  & 2.249e-03&  9.452e-04&  3.986e-04&  1.667e-04
& $\approx$ 1.3029(1.2600)\\
\addlinespace
$(0.5, 0.1, 0.9, 0.4, 0.5)$    &  6.605e-03 & 2.996e-03 & 1.322e-03 & 5.765e-04&  2.517e-04
  & $\approx$ 1.2353(1.2000)\\
\addlinespace
$(0.5,0.1,0.93,0.45,0.5)$    &  4.291e-03 &	1.802e-03&	7.329e-04&	2.960e-04&	1.196e-04	& $\approx$ 1.3448(1.3100)\\
\addlinespace
$(0.5,0.3,0.8,0.5,0.5)$    & 1.065e-02 &	5.227e-03 &	2.469e-03&	1.164e-03&	5.433e-04	  & $\approx$ 1.1552(1.1000)\\
\addlinespace
$(0.6,0.4,0.8,0.5,0.5)$    &  1.871e-02 &	9.577e-03	&4.610e-03&	2.176e-03&	1.017e-03	 & $\approx$ 1.1533(0.8333)\\ \bottomrule
	\end{tabular}
\end{table}

\section{Conclusion}

This paper studies a numerical analysis framework for the stochastic nonlinear time-space
fractional cable equation driven by rough noise with Hurst indices
$H_1,H_2 \in  (0, 1/2]$.   The model under consideration features a unique coupling of two Riemann–Liouville time-fractional derivatives and a spectral fractional Laplacian, which together capture the anomalous diffusive transport arising from temporal memory, spatial non-locality, and rough stochastic fluctuations in neuronal dynamics.
Utilizing an operator theoretic approach based on resolvent families and contour integral representations, we establish the well-posedness and spatial/temporal regularities of solutions. To handle the low regularity of the rough noise, we introduced a Wong-Zakai type approximation based on spectral truncation, and derived convergence rates for the regularized solution. For the fully discrete scheme, we combined the spectral Galerkin method for spatial discretization with the backward Euler convolution quadrature for temporal discretization. Extensive numerical experiments confirm the theoretical convergence rates and demonstrate that the observed empirical errors are consistent with the derived upper bounds. Our results show that the proposed method is  efficient for approximating solutions of this complex fractional stochastic model.

\section*{Acknowledgements} 

This study was funded by the Natural Science Foundation of Guangxi (2026GXNSFAA00640023), Hunan Provincial Natural Science Foundation  (2026JJ50360), and National Natural Science Foundation of China (12371198).
The authors would like to express our sincere gratitude to all editors and  reviewers for their careful reading of our manuscript and valuable suggestions that helped to improve this paper.

\end{document}